\documentclass[a4paper, 11pt]{article}
\usepackage{amsmath}
\usepackage{amssymb,esint}
\usepackage{amscd}
\usepackage{xspace}
\usepackage{fancyhdr}
\usepackage{color}
\usepackage{srcltx} 
\newtheorem{theorem}{Theorem}[section]

\newtheorem{definition}[theorem]{Definition}

\newtheorem{lemma}[theorem]{Lemma}

\newtheorem{proposition}[theorem]{Proposition}
\newtheorem{remark}[theorem]{Remark}

\newenvironment{proof}[1][Proof]{\textbf{#1.} }{\hfill\rule{0.5em}{0.5em}}
{\catcode`\@=11\global\let\AddToReset=\@addtoreset
\AddToReset{equation}{section}

\AddToReset{theorem}{section}

\begin{document}
\title{Pointwise gradient estimates for  a class of singular  quasilinear equation with measure data}
	\author{{\bf Quoc-Hung Nguyen\thanks{E-mail address:  qn2@nyu.edu, Department of Mathematics, New York University Abu Dhabi, Abu Dhabi, United Arab Emirates.}~~and~Nguyen Cong Phuc\thanks{E-mail address: pcnguyen@math.lsu.edu, Department of Mathematics, Louisiana State University, 303 Lockett Hall,
				Baton Rouge, LA 70803, USA. N. C. Phuc is supported in part by Simons Foundation, award number 426071. }}}
\date{}  
\maketitle
\begin{abstract}
Local and global pointwise gradient estimates are obtained for solutions to the quasilinear elliptic  equation with measure data $-\operatorname{div}(A(x,\nabla u))=\mu$
in a bounded and possibly nonsmooth domain $\Omega$ in $\mathbb{R}^n$. Here $\operatorname{div}(A(x,\nabla u))$ is modeled after the $p$-Laplacian. Our results extend earlier known results to the singular case in which $\frac{3n-2}{2n-1}<p\leq 2-\frac{1}{n}$.
  	
\end{abstract}   
                       \tableofcontents
 \section{Introduction and main results }

 In this paper,  the quasilinear elliptic equation with measure data 
 \begin{equation}\label{quasi-measure}
 -\operatorname{div}(A(x,\nabla u))=\mu
 \end{equation}
 is considered in a bounded open subset $\Omega$ of $\mathbb{R}^n$, $n\geq 2$. Here $\mu$ is a finite signed measure in $\Omega$ and
the nonlinearity  $A = (A_1, \dots, A_n):\mathbb{R}^n\times \mathbb{R}^n\to \mathbb{R}^n$ is vector valued function.
 Our main goal is to obtain pointwise estimates for gradients of solutions to equation  \eqref{quasi-measure} by means of nonlinear potentials of Wolff type. 
 To that end, throughout  the paper we  assume that $A=A(x,\xi)$ satisfies the following growth, ellipticity and  continuity assumptions: there exist $\Lambda\geq 1$, $1<p<2$, $s\geq 0$, and $\alpha\in (0, 2-p)$ such that
                                       \begin{equation}
                                       \label{condi1}| A(x,\xi)|\le \Lambda (s^2+|\xi|^2)^{(p-1)/2}, \quad | D_\xi A(x,\xi)|\le \Lambda (s^2+|\xi|^2)^{(p-2)/2},
                                       \end{equation}
                                        \begin{equation}
                                       \label{condi2}  \langle D_\xi A(x,\xi)\eta,\eta\rangle\geq \Lambda^{-1}  (s^2+|\xi|^2)^{(p-2)/2} |\eta|^2, 
                                       \end{equation}
                                       \begin{align}
                                        |D_\xi A(x,\xi)-D_\xi A(x,\eta)|\leq  & \Lambda (s^2+|\xi|^2)^{(2-p)/2} (s^2+|\eta|^2)^{(2-p)/2}\times  \nonumber \\
                                          & \qquad \times (s^2+|\xi|^2+|\eta|^2)^{(2-p-\alpha)/2}|\xi-\eta|^\alpha, \label{condi3}
                                       \end{align}
                                       and 
                                       \begin{equation}\label{Dini}
|A(x,\xi)-A(x_0,\xi)|\leq \Lambda \, \omega(|x-x_0|)(s^2+|\xi|^2)^{(p-1)/2}
                                       \end{equation}
 for every $x$ and  $x_0$ in $\mathbb{R}^n$ and every $(\xi,\eta)\in \mathbb{R}^n\times \mathbb{R}^n\backslash\{(0,0)\}$.  In \eqref{Dini},
  $\omega: [0,\infty)\to [0,1]$ is a non-decreasing function with $\omega(0)=0=\lim_{r\downarrow 0} \omega(r)$ and satisfies the Dini's condition:
                                          \begin{align}\label{dini}
                                          \int_{0}^{1}\omega(r)^{\gamma_0} \frac{dr}{r}=D<+\infty
                                          \end{align}
                                          for some $\gamma_0\in\left(\frac{n}{2n-1}, \frac{n(p-1)}{n-1}\right)$.

                                          A typical model for \eqref{quasi-measure} is obviously given by the $p$-Laplace equation with measure data
                                          \begin{equation}\label{p-Laplace}
                                          -\Delta_p \, u:= -{\rm div}(|\nabla u|^{p-2} \nabla u)=\mu \quad \text{in~} \Omega,
                                          \end{equation}
                                          or its nondegenerate version $(s>0)$: 
                                          $$-{\rm div}((|\nabla u|+ s^2)^{\frac{p-2}{2}} \nabla u)=\mu \quad \text{in~} \Omega.$$

                                          In this paper, we are concerned only with singular case in which 
                                          \begin{equation}\label{prestrict}
                                         \frac{3n-2}{2n-1} <p\leq 2-\frac{1}{n}.
                                          \end{equation}                   
                                          The case $p>2- \frac{1}{n}$ was considered in the  work \cite{Duzamin2, KM} (see also \cite{55DuzaMing, Mi2}) in which the authors  obtained  that if $u\in C^1(\Omega)$ solves \eqref{p-Laplace}  then it holds that 
                                          \begin{equation}\label{DMK}
                                          |\nabla u(x)| \leq C(n, p,\Lambda,D) \Big\{ [{\rm \bf I}^{R}_1(|\mu|)(x)]^{\frac{1}{p-1}} + \fint_{B_{R}(x)}|\nabla u(y)| dy\Big\}
                                          \end{equation}
                                          for every ball $B_R(x)\subset\Omega$ with $R\leq 1$. Here $\fint_E$ indicates the integral average over a measurable set $E$, and 
                                          $${\rm \bf I}^{R}_1(|\mu|)(x)=\int_0^R \frac{|\mu|(B_t(x))}{t^{n-1}}\frac{dt}{t}$$
                                          is a truncated first order Riesz's potential of $|\mu|$ at the point $x$.
                                          The restriction $p>2-1/n$ in \cite{Duzamin2, KM} has something to do with the fact that, in general, solutions to   \eqref{p-Laplace}  for a measure $\mu$ may not belong to  the Sobolev space $W^{1,  1}_{\rm loc}(\Omega)$ when $1<p\leq 2-1/n$. This is well known and can be seen by taking, e.g., $\mu$ to be the Dirac mass at a point.
                                          It also reveals that the linear potential ${\rm \bf I}^{R}_1(|\mu|)$ used in \eqref{DMK} may no longer be the right one when $1<p\leq 2-\frac{1}{n}$, and new ideas must be developed in order to attack this strongly singular case.

                                          In this paper, under the restriction \eqref{prestrict} we show that the solution  gradient can be pointwise controlled by the following (nonlinear) truncated Wolff's  potential 
                                             $${\bf P}^{R}_{\gamma}(|\mu|)(x) := \int_0^R \Big(\frac{|\mu|(B_t(x))}{t^{n-1}}\Big)^{\gamma}\frac{dt}{t}$$
                                          for certain $0<\gamma<1$. Note that ${\bf P}^{R}_{\gamma_1}(|\mu|) \leq C\, {\bf P}^{2R}_{\gamma_2}(|\mu|)$
                                          whenever $\gamma_1> \gamma_2>0$, and ${\bf I}_1^{R}(|\mu|) \leq C \, {\bf P}^{2R}_{\gamma}(|\mu|)^{\frac{1}{\gamma}}$ provided $0<\gamma<1$.
                                          
                                          Our main result is stated as follows.
                                          
                                          \begin{theorem}\label{grad1}
                                          	Let $\frac{3n-2}{2n-1}<p\leq 2-\frac{1}{n}$ and suppose that $u\in C^1(\Omega)$ solves \eqref{quasi-measure} for a finite measure $\mu$ in $\Omega$. Then under \eqref{condi1}-\eqref{dini} with $\gamma_0 \in\left(\frac{n}{2n-1}, \frac{n(p-1)}{n-1} \right)$
                                          	 we have 
                                          	\begin{equation}\label{Wpointwise}
                                          	|\nabla u(x)| \leq C \left\{ \Big[{\bf P}^{R}_{\gamma_0}(|\mu|)(x)\Big]^{\frac{1}{\gamma_0 (p-1)}} + \Big(\fint_{B_{R}(x)}(|\nabla u(y)|+s)^{\gamma_0} dy\Big)^{\frac{1}{\gamma_0}}\right\}
                                          	\end{equation}
                                          	for every ball $B_R(x)\subset\Omega$, where $C$ is a constant only depending on $n, p,\alpha,\Lambda,D,\gamma_0$.
                                          \end{theorem}
                                          
                                         The proof of Theorem \ref{grad1} is based on a new comparison estimate obtained in our recent work \cite{QH4} (see Lemma \ref{111120149} below), and the following  sharp quantitative $C^{1,\sigma}$ regularity estimate for the associated homogeneous equation which is interesting in its own right.
                                          
                                          \begin{theorem}\label{hol-Du} Suppose  that $A_0=A_0(\xi)$ is a vector field independent of $x$ and satisfies conditions \eqref{condi1}-\eqref{condi3}
                                          for some $s\geq 0$, $\Lambda\geq 1$, $1<p<2$ and $\alpha\in (0,2-p)$.	 Given any  $q\in (1, p+1)$, we define a vector field
                                          	$$U_q(\xi):= (s^2+ |\xi|^2)^{\frac{q-2}{2}} \xi, \qquad \xi\in\mathbb{R}^n.$$
                                          	Let $v\in W^{1,p}_{\rm loc}(\Omega)$ be a solution of ${\rm div}\, A_0(\nabla v)=0$ in $\Omega$.                                          	
                                          	Then there exist constants $C>1$ and $\sigma\in(0,1)$, only depending on $n,p,\alpha,\Lambda$, such that
                                          	\begin{align}\label{HU}
                                          	\fint_{B_{\rho}(x_0)} & |U_q(\nabla v) - [U_q(\nabla v)]_{B_{\rho}(x_0)}| \nonumber\\
                                          	& \leq C \left(\frac{\rho}{R}\right)^{\sigma(q-1)} \fint_{B_{R}(x_0)} |U_q(\nabla v) - [U_q(\nabla v)]_{B_{R}(x_0)}|
                                          	\end{align}
                                          	for every $B_R(x_0)\subset\Omega$ and $\rho<R$.
                                          \end{theorem}

We notice that Theorem \ref{hol-Du} generalizes the result \cite{Diening} in which  the case $q=p$ was considered in a slightly different context.  In our proof of Theorem  \ref{grad1}, Theorem   \ref{hol-Du} will be used with $q=1+\gamma_0$, where $\gamma_0\in \left(\frac{n}{2n-1}, \frac{n(p-1)}{n-1} \right)$. 
We remark that Theorem \ref{hol-Du} also holds in the case $p> 2$ provided the condition \eqref{condi3} is replaced by the condition
$$
 |D_\xi A(x,\xi)-D_\xi A(x,\eta)|\leq   \Lambda (s^2+|\xi|^2+|\eta|^2)^{(p-2-\alpha)/2}|\xi-\eta|^\alpha 
$$
for some $\alpha\in(0,p-2)$. For $p>2$, see also \cite[Theorem 3.1]{Duzamin2} where the case $q=\frac{p+2}{2}$ is considered.

The condition $u\in C^1(\Omega)$ in Theorem \ref{grad1} is by no means essential. In fact, it is enough to assume $u\in W^{1,p}_{\rm loc}(\Omega)$ in which case the pointwise bound \eqref{DMK} holds for any Lebesgue point $x$ of the vector function $(s^2+|\nabla u|^2)^{\frac{\gamma_0-1}{2}}\nabla u$. 
Moreover, by approximation the pointwise bound \eqref{DMK} also holds a.e. for any distributional  solution $u$ to the Dirichlet problem 
\begin{equation}\label{Dirich}
\left\{ \begin{array}{rcl}
- \operatorname{div}\left( {A(x,\nabla u)} \right) &=& \mu \quad \text{in} \ \Omega, \\ 
u &=& 0 \quad \text{on} \ \partial \Omega,  
\end{array} \right.
\end{equation}
provided $u$ satisfies the following additional properties:

\noindent ({\bf P1}). For each $k>0$ the truncation $T_k(u)$ belongs to $W_0^{1,p}(\Omega)$, where we define
$$T_k(s)=\max\{\min\{s,k\},-k\}, \qquad  s\in\mathbb{R}.$$

\noindent ({\bf P2}). For each $k>0$ there exsits a finite signed measure $\mu_k$ in $\Omega$ such that 
$$- \operatorname{div}\left( {A(x,\nabla T_{k}(u))} \right) = \mu_k \quad  \text{in} \  \mathcal{D}'(\Omega),$$
and if we set $|\mu_k|(\mathbb{R}^n\setminus\Omega)=|\mu|(\mathbb{R}^n\setminus\Omega)=0$ then it holds that 
$\mu_k \rightarrow \mu$ and $|\mu_k| \rightarrow |\mu|$ weakly as measures in $\mathbb{R}^n$.

We recall that if $u$ is a measurable function  in $\Omega$, finite a.e., and satisfying the above two conditions then there exists (see \cite[Lemma 2.1]{bebo}) a unique measurable function $v:\Omega\to \mathbb{R}^n$ such that $\nabla T_k(u)=v\, \chi_{\{|u|\leq k\}}$ 
 a.e. in $\Omega$  for each $k>0$. We define the gradient $\nabla u$ of $u$ by $\nabla u=v$ and accordingly  $\nabla u$  in \eqref{Dirich} should be understood in this sense. Note that 
 if $v$ belongs to $L^q(\Omega)^n$, $1\leq q\leq p$, then $u\in W^{1,q}_{0}(\Omega)$ and $v$ coincides with the distributional gradient of $u$ (see \cite[Remark 2.10]{11DMOP}).
We mention that if, e.g., $u$ is a {\it renormalized solution} to \eqref{Dirich} (see \cite{11DMOP}) then $u$ satisfies the above two properties.

In fact, for solutions $u$ of \eqref{Dirich} satisfying ({\bf P1}) and ({\bf P2})  we can obtain pointwise a.e. estimates up to the boundary of $\Omega$ provided  $\partial\Omega$
is sufficiently flat (in the sense of Reifenberg).

\begin{definition}We say that $\Omega$ is a $(\delta,R_0)$-Reifenberg flat domain for $\delta\in (0,1)$ and $R_0>0$ if for every $x\in\partial \Omega$ and every $r\in(0,R_0]$, there exists a system of coordinates $\{z_1,z_2,...,z_n\}$, which may depend on $r$ and $x$, so that in this coordinate system $x=0$ and that 
$$
	B_r(0)\cap \{z_n>\delta r\}\subset B_r(0)\cap \Omega\subset B_r(0)\cap\{z_n>-\delta r\}.
$$
\end{definition}

We notice that this class of  domains is rather wide since it includes $C^1$ domains, Lipschitz domains with sufficiently small Lipschitz constants, and even certain fractal domains. Besides, it has many important roles in the theory of minimal surfaces and free boundary problems. This class appeared first in the work of Reifenberg \cite{55Re} in the context of Plateau problems. Many of the properties of Reifenberg flat domains can be found in \cite{55KeTo1,55KeTo2}. 

Our pointwise  estimates up to the boundary of $\Omega$ read as follows.
                     
                \begin{theorem} \label{boundary}  Let $\frac{3n-2}{2n-1}<p\leq 2-\frac{1}{n}$ and suppose that $u$ is a solution of \eqref{Dirich} that satisfies properties $({\bf P1})$ and $({\bf P2})$.
                 Then under \eqref{condi1}-\eqref{dini} for any $\kappa\in (0,1/2)$, there exits $\delta>0$ 
                  such that if  $\Omega$ is a $(\delta,R_0)$-Reifenberg flat domain for some $R_0>0$ then    we have 
                                         \begin{align}\label{ine3}
                                        |\nabla u(x)|\leq C d(x)^{-\kappa} \left(\Big[{\bf P}^{2{\rm diam}(\Omega)}_{\gamma_0}(|\mu|)(x)\Big]^{\frac{1}{\gamma_0 (p-1)}}+s\right)
                                        \end{align}
                                        for a.e. $x\in \Omega$.  Here $\gamma_0$ is any number  in $\left(\frac{n}{2n-1}, \frac{n(p-1)}{n-1} \right)$ and $d(x)$ is the distance from 
                                        $x$ to the boundary of $\Omega$. 
                                      \end{theorem}    
              
          We notice that due to the potential irregularity of $\Omega$, it is not possible to take $\kappa=0$ in  \eqref{ine3} in general.

\section{Sharp quantitative $C^{1,\sigma}$ regularity estimates}  

This section is devoted to the proof of Theorem \ref{hol-Du}. We first recall the following   basic inequalities that were proved in \cite[Lemmas 2.1 and 2.2]{AF}:
\begin{equation}\label{ieq1}
1\leq \frac{\int_{0}^{1}(s^2+|\xi_1+t(\xi_2-\xi_1)|^2)^\gamma dt}{(s^2+|\xi_1|^2+|\xi_2|^2)^\gamma}\leq \frac{8}{2\gamma+1},
\end{equation}
and
\begin{equation}\label{ieq2}
(2\gamma+1) |\xi_1-\xi_2| \leq \frac{|(s^2+|\xi_1|^2)^\gamma\xi_1-(s^2+|\xi_2|^2)^\gamma\xi_2|}{(s^2+|\xi_1|^2+|\xi_2|^2)^\gamma}\leq \frac{C(n)}{2\gamma+1} |\xi_1-\xi_2|,
\end{equation}
which hold for any $\xi_1,\xi_2\in \mathbb{R}^n$, $s\geq 0$, and  $\gamma\in (-1/2,0)$.

For $s\geq  0$, we let  
$$Z(\xi)= (s^2 +|\xi|^2)^{\frac{1}{2}}, \qquad \xi\in \mathbb{R}^n,$$
and define 
$$H(\xi)=Z(\xi)^p, \qquad  V(\xi)= Z(\xi)^{(p-2)/2} \xi, \qquad \xi\in \mathbb{R}^n.$$

Then the conditions \eqref{condi1}-\eqref{condi3} imposed on $A_0$ in Theorem \ref{hol-Du} can be restated as 
\begin{align}\label{A01}
A_0(\xi)|\le \Lambda Z(\xi)^{p-1}, \quad | D A_0(\xi)|\le \Lambda Z(\xi)^{p-2},
\end{align}
\begin{align}\label{A02}
\langle D A_0(\xi)\eta,\eta\rangle\geq \Lambda^{-1}  Z(\xi)^{p-2} |\eta|^2,
\end{align}
and 
\begin{equation}\label{A03}
|D A_0(\xi)-D A_0(\eta)|\leq \Lambda\, Z(\xi)^{p-2} Z(\eta)^{p-2} (s^2+|\xi|^2+|\eta|^2)^{(2-p-\alpha)/2}|\xi-\eta|^\alpha
\end{equation}
for some $\Lambda\geq 1$, $\alpha\in(0, 2-p)$, and for every  $(\xi,\eta)\in \mathbb{R}^n\times \mathbb{R}^n\backslash\{(0,0)\}$.

It follows from \eqref{A02} that the following strict monotonicity holds
\begin{equation}\label{esz}(A_0(\xi)-A_0(\eta))\cdot(\xi-\eta)\geq c(p,\Lambda) (s^2 + |\xi|^2 +|\eta|^2)^{\frac{p-2}{2}}|\xi-\eta|^2 
\end{equation}
for all $(\xi,\eta)\in \mathbb{R}^n\times \mathbb{R}^n$. 
Moreover, by the second inequality in \eqref{A01} we have 
\begin{align*}
|A_0(\xi)-A_0(\eta)|&=\left|\int_0^1 D A_0(t\xi +(1-t)\eta) (\xi-\eta) dt\right|\\
&\leq \Lambda |\xi-\eta| \int_0^1 Z(t\xi +(1-t)\eta)^{p-2}  dt\\
& \leq C |\xi-\eta| (s^2 + |\eta|^2 + |\xi-\eta|^2)^{\frac{p-2}{2}},
\end{align*}
where we used \eqref{ieq1} in the least inequality. 	Thus we get 
\begin{equation}\label{Aandxi}
(A_0(\xi)-A_0(\eta))\cdot(\xi-\eta)\simeq (s^2 + |\xi|^2 +|\eta|^2)^{\frac{p-2}{2}}|\xi-\eta|^2,
\end{equation}
and 
\begin{equation}\label{Aonly}
|A_0(\xi)-A_0(\eta)| \simeq (s^2 + |\xi|^2 +|\eta|^2)^{\frac{p-2}{2}}|\xi-\eta|.
\end{equation}

Let $v\in W^{1,p}_{\rm loc}(\Omega)$ be a solution of ${\rm div} \, A_0(\nabla v) = 0 \quad {\rm in}\quad   \Omega$, 
i.e.,
\begin{equation}\label{testf}  
\int_{\Omega} A_{0i}(\nabla v) D_i \phi=0,
\end{equation}
for every $\phi\in C_0^\infty(\Omega)$, where $A_{0i}$ is the $i^{\text{th}}$ component of $A_0$.  We observe that in order to prove  \eqref{HU} for $v$, by a standard approximation (see, e.g., \cite{Duzamin1}), we may assume that $s>0$. 

Then by \cite[Theorem 8.1]{Giu} and \cite[Proposition 8.1]{Giu}, $v$ has second derivatives $D^2 v$ and $H(\nabla v)\in W^{1,2}_{\rm loc}(\Omega)$, such that for every subset $\Sigma\Subset \Omega$, we have   
\begin{equation}\label{D2vbound}
\int_\Sigma Z(\nabla v)^{p-2} |D^2 v|^2 \leq C(\Sigma, s) \int_\Omega H(\nabla v), 
\end{equation}
and 
$$\int_\Sigma |\nabla [H(\nabla v)]|^2 \leq C(\Sigma, s) \int_\Omega |H(\nabla v)|^2. $$

In \eqref{testf}, taking $\phi=D_k \varphi$,  $\varphi\in C_0^\infty(\Omega)$, and integrating by parts,
we find  
\begin{equation}\label{Dsv-sol}
\int A_{ij}(\nabla v) D_{jk} vD_i \varphi=0,
\end{equation}
where we set
$$A_{ij}(\xi)=\frac{\partial A_{0i}(\xi)}{\partial \xi_j}, \quad i,j=1, \dots, n.$$

By \eqref{D2vbound}, for each $k\in \{1, \dots,n\}$, the function $\varphi (D_k v-b_k)$, $\varphi\in C_0^\infty(\Omega), b_k\in \mathbb{R}$, is a valid   test function for \eqref{Dsv-sol}, and thus we find
$$\int A_{ij}(\nabla v) D_{jk} v D_{ik} v \varphi + \int A_{ij}(\nabla v) D_{jk} v (D_{k}v -b_k)  D_i\varphi =0. $$

Now observe that $D_j [H(\nabla v)]= p Z(\nabla v)^{p-2} D_{jk}v D_k v$ and thus when $(b_1, \dots, b_n)=(0,\dots,0)$ the last equality can be written as 
$$p\int A_{ij}(\nabla v) D_{jk} v D_{ik} v \varphi + \int a_{ij}(\nabla v) D_{j} [H(\nabla v)]  D_i\varphi =0, $$
where $a_{ij}(\nabla v)= Z(\nabla v(x))^{2-p} A_{ij}(\nabla v(x))$, a uniformly elliptic matrix.

In view of \eqref{A02}, this gives
\begin{equation}\label{subsolforH}
\int a_{ij}(\nabla v) D_{j} [H(\nabla v)]  D_i\varphi  \leq -c \int Z(\nabla v)^{p-2} |D^2 v|^2 \varphi  \leq -c \int |D V(\nabla v)|^2 \varphi,
\end{equation}
for all $\varphi\in C_0^\infty(\Omega)$, $\varphi\geq 0$. 

In particular, $H(\nabla v) \in W^{1,2}_{\rm loc}(\Omega)$ is a subsolution to a uniformly elliptic equation 
in divergence form, which yields that $H(\nabla v) \in L^{\infty}_{\rm loc}(\Omega)$ with the estimate
\begin{equation}\label{Linf}
\sup_{B_{R/2}} H(\nabla v) \leq C \fint_{B_R} H(\nabla v), \qquad \forall B_R\subset\Omega.
\end{equation}

In what follows, for any ball $B_r(x_0)\subset\Omega$ we denote by $\Phi(x_0, r)$ the excess functional
 $$\Phi(x_0, r):=\fint_{B_{r}(x_0)} |V(\nabla v) - [V(\nabla v)]_{B_{r}(x_0)}|^2.$$
We also set 
$$M(r)=\sup_{B_r(x_0)} H(\nabla v).$$

With \eqref{subsolforH}, one can now argue as in the proof of \cite[Proposition 3.1]{GM} to obtain the following result. 
\begin{lemma}\label{PhiandH} There is a consntant $c>0$ independent of $s$ such that 
	\begin{equation}\label{z5}
	\Phi(x_0, R/2) \leq c \Big(M(R) - M(R/2)\Big),
	\end{equation}
	for every $B_R(x_0)\Subset\Omega$.
\end{lemma}
\begin{proof}  By \eqref{subsolforH}, the function $\mathbf{v}(x):=M(R)-H(\nabla v(x))$ is a nonnegative supersolution in $B_{R}(x_0)$ of the uniformly elliptic equation $\partial_{i}(a_{ij}(\nabla v)\partial_{j} u)= 0$. Thus, by the weak Harnack inequality, we have 
	\begin{align}\label{Vest}
	\fint_{B_{R}(x_0)} \mathbf{v}(x)dx\leq C\inf_{B_{R/2}(x_0)} \mathbf{v}\leq C(M(R)-M(R/2)).
	\end{align}

	Let $\chi\in W^{1,2}_0(B_R(x_0))$ be the weak solution to  
	\begin{align*}
	\int_{B_R(x_0)} a_{ij}(\nabla v) \partial_{i} \chi   \partial_j\varphi =\frac{1}{R^2}\int_{B_R(x_0)}\varphi dx ~~\forall\varphi\in W^{1,2}_0(B_R(x_0)).
	\end{align*}
	Then taking $\varphi= \chi \mathbf{v}$ as test function for the above equation, we get 
	\begin{align*}
	\frac{1}{2}\int_{B_R(x_0)} a_{ij}(\nabla v) \partial_{i} \chi^2 \partial_j \mathbf{v}\leq \frac{1}{R^2}\int_{B_R(x_0)} \chi \mathbf{v} dx.
	\end{align*}
	Now, taking $\varphi=\chi^2$ as test function for \eqref{subsolforH}, we find
	\begin{align*}
	& \int_{B_{R}(x_0)} |\nabla V(\nabla v)|^2 \chi^2 \leq 	-C\int_{B_{R}(x_0)} a_{ij}(\nabla v) \partial_{j} [H(\nabla v)]  \partial_i\chi^2 \\
        &  =C\int_{B_{R}(x_0)} a_{ij}(\nabla v) \partial_{j} \textbf{v}  \partial_i\chi^2  \leq \frac{C}{R^2} \int_{B_R(x_0)} \chi \mathbf{v} dx  \leq \frac{C}{R^2} \int_{B_R(x_0)}\mathbf{v} dx,
	\end{align*}
where we used  the fact that $\|\chi\|_{L^\infty(B_{R}(x_0))}\leq C$ (by homogeneity) in the last inequality. Also, by homogeneity and  the weak Harnack inequality we  have that $\inf_{B_{R/2}(x_0)} \chi \geq c>0$ and thus  combining with \eqref{Vest} we obtain 

\begin{align*}
	& \int_{B_{R/2}(x_0)} |\nabla V(\nabla v)|^2 \leq \frac{C}{R^2}\left(M(R)-M(R/2)\right).
	\end{align*}

Finally, we use   Poincar\'e's inequality  in the last bound to obtain  \eqref{z5}. This completes the proof of the lemma.  
\end{proof}

The following lemma can be proved by adapting the proof of \cite[Lemma 2.9]{AF}  to our setting.
\begin{lemma}\label{ReverseforD}
	Let $B_R(x_0)\Subset \Omega$ and suppose  that	$\sup_{B_R(x_0)} |\nabla v|^2 \leq c (s^2+ |\xi|^2)$
	for some $c>0$ and $\xi\in\mathbb{R}^n$. Then there exist $C, \delta>0$ independent of $s$, $\xi$, and $B_R(x_0)$ such that
	\begin{equation}\label{ReverseforD-}\fint_{B_{R/2}(x_0)} |\nabla v-\xi|^{2+2\delta} \leq C \left(\fint_{B_R(x_0)} |\nabla v-\xi|^{2}\right)^{1+\delta}.
	\end{equation}
\end{lemma}
\begin{proof} For $B_{\rho}(y_0)\subset B_{R}(x_0)$ we set 
	\begin{equation}\label{vtil}
	\tilde{v}=v(x)-[v]_{B_{\rho}(y_0)}-\xi\cdot(x-y_0),
	\end{equation} 
	and let $\varphi$ be a function in $C^\infty_c(B_\rho(y_0))$ such that  $0\leq \varphi\leq 1$, $\varphi=1$ in $B_{\rho/2}(y_0)$ and $|\nabla \varphi|\leq C/\rho$. Note that 
	\begin{equation*}
	\int \left(A_{0}(\nabla v)-A_{0}(\xi)\right) \cdot \nabla (	\tilde{v}\varphi^2)=	\int A_{0}(\nabla v) \cdot \nabla (	\tilde{v}\varphi^2)=0,
	\end{equation*}
	and thus  
	\begin{equation*}
	\int  \left(A_{0}(\nabla v)-A_{0}(\xi)\right) \cdot \nabla \tilde{v} \varphi^2=-	2\int \left(A_{0}(\nabla v)-A_{0}(\xi)\right) \cdot \nabla\varphi \tilde{v} \varphi.
	\end{equation*}
	
     Then by \eqref{esz} and \eqref{Aonly}, 
	\begin{equation*}
	\int (s^2 + |\nabla v|^2 +|\xi|^2)^{\frac{p-2}{2}}|\nabla \tilde{v}|^2\varphi^2\leq C	\int(s^2 + |\nabla v|^2 +|\xi|^2)^{\frac{p-2}{2}}|\nabla \tilde{v}||\nabla\varphi| |\tilde{v}||\varphi|, 
	\end{equation*}
 which by	 H\"older's inequality yields 
	\begin{equation*}
	\int (s^2 + |\nabla v|^2 +|\xi|^2)^{\frac{p-2}{2}}| \nabla \tilde{v}|^2\varphi^2\leq C	\int(s^2 + |\nabla v|^2 +|\xi|^2)^{\frac{p-2}{2}}|\nabla\varphi|^2 |\tilde{v}|^2.
	\end{equation*}

	Note that  $\sup_{B_R(x_0)} |\nabla v|^2 \leq c (s^2+ |\xi|^2)$ implies that 
	\begin{align*}
	(s^2 + |\nabla v|^2 +|\xi|^2)^{\frac{p-2}{2}}\simeq (s^2 +|\xi|^2)^{\frac{p-2}{2}},
	\end{align*}
	and thus using the property of $\varphi$ we find
	\begin{equation*}
	\int_{B_{\rho/2}(y_0)} |\nabla v-\xi|^2 dx \leq \frac{C}{\rho^2}	\int_{B_\rho(y_0)} |\tilde{v}|^2 dx .
	\end{equation*}

	Now using   Sobolev-Poincar\'e's inequality (note that $[\tilde{v}]_{B_\rho(y_0)}=0$) and Gehring lemma on higher integrability, we get \eqref{ReverseforD-} as desired.
\end{proof}\\

We can now use Lemma \ref{ReverseforD} and argue as in the proof of \cite[Lemma 2.10]{AF} to deduce the following important result. We remark this is where we use the assumption \eqref{A03} on $A_0(\xi)$.
\begin{lemma}\label{important} Under \eqref{A03}, there is a constant $C>0$, independent of $s$, such that for every $\tau\in(0,1)$ there exists $\epsilon>0$, independent of $s$, such that 
	\begin{equation}\label{z6}
	\Phi(x_0,R)\leq \epsilon \sup_{B_{R/2}(x_0)} H(\nabla v) \qquad \Rightarrow \qquad  \Phi(x_0,\tau R)\leq C \tau^2 \Phi(x_0,R)
	\end{equation}
	for every $B_R(x_0)\Subset\Omega$.	 
\end{lemma}
\begin{proof} Take $\xi\in \mathbb{R}^n$ such that 
$V(\xi)=[V(D v)]_{B_R(x_0)}$. Then by \eqref{Linf}, 
	\begin{align*}
	\sup_{B_{R/2}(x_0)} H(\nabla v)& \leq C \fint_{B_R(x_0)} H(\nabla v)\leq \fint_{B_R(x_0)} \left(s^p+|V(\nabla v)|^2\right)\\&\leq C\left(s^p+\Phi(x_0,R)+|V(\xi)|^2\right).
	\end{align*}

	Thus, if $\varepsilon<1/2C$ we deduce
	\begin{equation}\label{EQ2b}
	\Phi(x_0,R)\leq 2C \varepsilon(s^p+V(\xi)|^2)\leq C\varepsilon (s^2+|\xi|^2)^{p/2},
	\end{equation}
	and hence
	\begin{equation}\label{EQ2}
	\sup_{B_{R/2}(x_0)}|\nabla v|^p \leq   \sup_{B_{R/2}(x_0)} H(\nabla v)\leq C (s^2+|\xi|^2)^{p/2}.
	\end{equation}

	Let $\tilde{v}$ be as in \eqref{vtil}, and let $v_0\in \tilde{v}+W_0^{1,2}(B_{R/4})$ be the solution of 
	\begin{equation}\label{v0equa}
	\int_{B_{R/4}} A_{ij}(\xi)\partial_{j} v_0\partial_{i}\varphi =0 \qquad \forall\varphi\in W^{1,2}_0(B_{R/4}).
	\end{equation}
	Since  $C^{-1}Z(\xi)^{p-2}\mathbb{I}_n\leq (A_{ij})\leq CZ(\xi)^{p-2}\mathbb{I}_n$, by the standard regularity we get 
	\begin{align}\label{Ineq1}
	\fint_{B_{\tau R}}|\nabla v_0-[\nabla v_0]_{B_{\tau R}}|^2dx\leq C\tau^2 \fint_{B_{R/4}}|\nabla v_0-[\nabla v_0]_{B_{R/4}}|^2dx,
	\end{align}
for every $\tau\in (0,1/4)$.

	Let $\varphi\in W^{1,2}_0(B_{R/4})$. Using the relation 
$$\int_{B_{R/4}} \left(A_{0}(\nabla v)-A_{0}(\xi)\right) \cdot \nabla \varphi=0,$$
 we can write 
$$
	\int_{B_{R/4}}\int_{0}^{1} A_{ij}(\xi+t \nabla \tilde{v}) dt  \partial_{j}  \tilde{v}\partial_i \varphi=0.
$$
	
Combining this with \eqref{v0equa} we have 
$$
	\fint_{B_{R/4}}\int_{0}^{1} \left(A_{ij}(\xi+t \nabla \tilde{v})-A_{ij}(\xi)\right) dt\partial_{j}\tilde{v}\partial_i \varphi dx=  \fint_{B_{R/4}} A_{ij}(\xi)\left(\partial_{j} v_0-\partial_{j}\tilde{v}\right)\partial_{i}\varphi dx.
$$

	Then choosing $ \varphi=v_0-\tilde{v}$ as a test function, we get 
	\begin{align*}
	Z(\xi)^{p-2} &\fint_{B_{R/4}} |\nabla (v_0-\tilde{v})|^2dx\leq \\
      &C \fint_{B_{R/4}}\int_{0}^{1}| \left(A_{ij}(\xi+t \nabla \tilde{v})-A_{ij}(\xi)\right)| dt|\nabla \tilde{v}| |\nabla (v_0-\tilde{v})| dx.
	\end{align*}

	On the other hand, thanks to \eqref{A03} we find that 
	\begin{align*}
	\int_{0}^{1} &| \left(A_{ij}(\xi+t\nabla\tilde{v})-A_{ij}(\xi)\right)| dt \\
&\leq CZ(\xi)^{p-2} \int_{0}^{1}  Z(\xi+t\nabla\tilde{v})^{p-2} (s^2+|\xi|^2+|\xi+t\nabla\tilde{v}|^2)^{(2-p-\alpha)/2}|t \nabla \tilde{v}|^\alpha dt\\& \leq 
	CZ(\xi)^{-\alpha} | \nabla \tilde{v}|^\alpha \int_{0}^{1}  Z(\xi+t\nabla\tilde{v})^{p-2}  dt \qquad (\text{by \eqref{EQ2}})\\
&\leq C Z(\xi)^{-\alpha} | \nabla \tilde{v}|^\alpha (s^2+|\xi|^2+|\nabla\tilde{v}|^2)^{(p-2)/2} \qquad (\text{by \eqref{ieq1}})
	\\&\leq CZ(\xi)^{p-2-\alpha}|\nabla \tilde{v}|^\alpha.
	\end{align*}

     Thus, 
$$
	\fint_{B_{R/4}} |\nabla (v_0-\tilde{v})|^2dx\leq C 	Z(\xi)^{-\alpha}\fint_{B_{R/4}}|\nabla \tilde{v}|^{1+\alpha} |\nabla (v_0-\tilde{v})| dx,
$$
which by H\"older's inequality yields 
$$
	\fint_{B_{R/4}} |\nabla (v_0-\tilde{v})|^2dx\leq C 	Z(\xi)^{-2\alpha}\fint_{B_{R/4}}|\nabla \tilde{v}|^{2+2\alpha}dx.
$$
	
        For any $0<\delta\leq \alpha$, by \eqref{EQ2}, 
	\begin{align*}
	\fint_{B_{R/4}} |\nabla (v_0-\tilde{v})|^2dx &\leq C 	Z(\xi)^{-2\delta}\fint_{B_{R/4}}|\nabla \tilde{v}|^{2+2\delta}dx\\
     &=C 	Z(\xi)^{-2\delta}\fint_{B_{R/4}}|\nabla v-\xi|^{2+2\delta}dx.
	\end{align*}

	Hence by \eqref{EQ2} and  Lemma \ref{ReverseforD}, we obtain 
	\begin{align}\label{z2}
	\fint_{B_{R/4}} |\nabla (v_0-\tilde{v})|^2dx\leq C 	Z(\xi)^{-2\delta}\left(\fint_{B_{R/2}}|\nabla v-\xi|^{2}dx\right)^{1+\delta}
	\end{align}
	for some $0<\delta\leq \alpha$.
	
          Note that by \eqref{ieq2}, 
	\begin{align*}
	\Phi(x_0,\tau R)&\leq C \fint_{B_{\tau R}}|V(\nabla v)-V([\nabla v]_{B_{\tau R}})| dx\\ &\leq C
	\fint_{B_{\tau R}}\left(s^2+|\nabla v|^2+|[\nabla v]_{B_{\tau R}}|^2\right)^{\frac{p-2}{2}}|\nabla v-[\nabla v]_{B_{\tau R}}|^2 dx \\ &\leq C \left(s^2+|[\nabla v]_{B_{\tau R}}|^2\right)^{\frac{p-2}{2}}
	\fint_{B_{\tau R}}|\nabla \tilde{v}-[\nabla \tilde{v}]_{B_{\tau R}}|^2 dx.
	\end{align*}

	Using \eqref{Ineq1}, for any  $\tau\in (0,1/4)$ we get 
	\begin{align*}
	\fint_{B_{\tau R}} &|\nabla \tilde{v}-[\nabla \tilde{v}]_{B_{\tau R}}|^2 dx\leq  2\fint_{B_{\tau R}}|\nabla v_0-[\nabla  v_0]_{B_{\tau R}}|^2 +|\nabla \tilde{v}-\nabla v_0|^2 dx\\&\leq  C\tau^2\fint_{B_{ R/4}}|\nabla v_0-[\nabla  v_0]_{B_{ R/4}}|^2 +C\tau^{-n}\fint_{B_{R/4}}|\nabla \tilde{v}-\nabla v_0|^2 dx 
	\\&\leq  C\tau^2\fint_{B_{ R/4}}|\nabla \tilde{v}-[\nabla  \tilde{v}]_{B_{ R/4}}|^2 +C\tau^{-n}\fint_{B_{R/4}}|\nabla \tilde{v}-\nabla v_0|^2 dx 
	\\&\leq  C\tau^2\fint_{B_{ R/2}}|\nabla v-\xi|^2 +C\tau^{-n}Z(\xi)^{-2\delta}\left(\fint_{B_{R/2}}|\nabla v-\xi|^{2}dx\right)^{1+\delta},
	\end{align*}
	where we used \eqref{z2} in the last inequality.

     On the other hand, by \eqref{ieq2} and \eqref{EQ2}, 
	\begin{align}\nonumber
	\fint_{B_{ R/2}}|\nabla v-\xi|^2 &\leq C\fint_{B_{ R/2}}\left(s^2+|\xi|^2+|\nabla v|^2\right)^{\frac{2-p}{2}} |V(\nabla v)-V(\xi)|^2 dx\\&\leq C(s^2+|\xi|^2)^{\frac{2-p}{2}}\Phi(x_0,R)\label{z3}.
	\end{align}

	Hence, 
	\begin{align*}
	\Phi(x_0,\tau R)&\leq C \left(\frac{s^2+|\xi|^2}{s^2+|[\nabla v]_{B_{\tau R}}|^2}\right)^{\frac{2-p}{2}}\times\\
	&\qquad \qquad \times \left(\tau^2\Phi(x_0,R)+\tau^{-n}Z(\xi)^{-\delta p}\Phi(x_0,R)^{1+\delta}\right),
	\end{align*}
	which by \eqref{EQ2b} yields $$
	\Phi(x_0,\tau R)\leq C \left(\frac{s^2+|\xi|^2}{s^2+|[\nabla v]_{B_{\tau R}}|^2}\right)^{\frac{2-p}{2}}
	\left(\tau^2+\tau^{-n}\varepsilon^\delta\right)\Phi(x_0,R).$$

	Now, we show that for $\varepsilon>0$ small enough, 
	\begin{align}\label{z4}
	|\xi|^2\leq C(s^2+|[\nabla v]_{B_{\tau R}}|^2).
	\end{align}
	Indeed, 
	\begin{align*}
	|\xi|^2&\leq 2\left(|\xi-[\nabla v]_{B_{\tau R}}|^2+|[\nabla v]_{B_{\tau R}}|^2\right)\\& \leq C\left(\fint_{B_{\tau R}}|\nabla v-\xi|^2+|[\nabla v]_{B_{\tau R}}|^2\right)
	\\& \leq C\left(\tau^{-n}\fint_{B_{R/2}}|\nabla v-\xi|^2+|[\nabla v]_{B_{\tau R}}|^2\right)
	\\& \leq C\left(\tau^{-n}(s^2+|\xi|^2)^{\frac{2-p}{2}}\Phi(x_0,R)+|[\nabla v]_{B_{\tau R}}|^2\right) \qquad (\text{by }  \eqref{z3})
	\\& \leq C\left(\tau^{-n}\varepsilon (s^2+|\xi|^2)+|[\nabla v]_{B_{\tau R}}|^2\right) \qquad (\text{by }  \eqref{EQ2b}).
	\end{align*}

	Thus if $C\tau^{-n}\varepsilon\leq 1/2$,  we obtain \eqref{z4}. Therefore, we get \eqref{z6}  if we further restrict $\varepsilon$ so that  $\varepsilon<\tau^{\frac{n+2}{\delta}}$.
\end{proof}

Lemmas \ref{PhiandH} and \ref{important}  yield the following alternative result.

\begin{lemma}\label{leau} Let $\tau,\varepsilon\in (0,1/4)$ be fixed as in Lemma \ref{important} such that $C\tau^2<\tau$, where $C$ is the constant in \eqref{z6}.  There exists $\delta=\delta(\tau)\in (0,1)$ such that either 
$$
	\Phi(x_0,\tau R)\leq \tau \Phi(x_0, R),$$
	or 
$$
	\Phi(x_0,R)\geq \varepsilon M(R/2)~~\text{and}~~M(R/4)\leq \delta M(R/2)$$
	provided $B_R(x_0)\Subset\Omega$.
\end{lemma}
\begin{proof} 
	If $\Phi(x_0,R)< \varepsilon M(R/2)$, then by Lemma \ref{important} we get $ \Phi(x_0,\tau R)\leq \tau \Phi(x_0,R)$. If  $M(R/4)> \delta M(R/2)$ and 	$\Phi(x_0,R)\geq \varepsilon M(R/2)$, then by Lemma \ref{PhiandH}, 
	\begin{align*}
	\Phi(x_0,R/4) & \leq C[M(R/2) - M(R/4)]\\
	&\leq C(1-\delta)M(R/2)\leq C(1-\delta)\varepsilon^{-1}\Phi(x_0,R).
	\end{align*}
	
	Thus, 
$$
	\Phi(x_0,\tau R) \leq C(\tau)(1-\delta)\varepsilon^{-1}\Phi(x_0,R).
$$
	
	Now choosing $\delta\in (0,1)$ such that $C(\tau)(1-\delta)\varepsilon^{-1}<\tau$ we get the result.
\end{proof}

We next follow an alternative argument  in the spirit of \cite[Theorem 3.1]{GM} to derive a decay estimate for the excess functional $\Phi(x_0,r)$.


\begin{theorem}\label{calphaforphi} Suppose  that $A_0$ satisfies \eqref{A01}, \eqref{A02}, and \eqref{A03}. 
	There exist constants $C>1$ and $\sigma_1\in(0,1)$, both independent of $s$, such that
	\begin{equation*}
	\Phi(x_0,\rho)\leq C \left(\frac{\rho}{R}\right)^{2\sigma_1} \Phi(x_0,R)
	\end{equation*}
	for every $B_R(x_0)\subset\Omega$ and $\rho<R$. 
\end{theorem}
\begin{proof} For ease of notation, we shall drop $x_0$ and  write $\Phi(x_0,r)$ as $\Phi(r)$. Let $\tau,\varepsilon$ and $\delta$ be  as in Lemma \ref{leau}. Let $k,h\in \mathbb{N}$ be such that $
	\delta^k\varepsilon<\tau$ and $\tau^{(h/k)-1}\varepsilon^{-2}<\tau$. Also, let $
	r_j=\tau^{jh}R$ and $\rho_j=\tau^{j}R$.
	It is enough to show that
$$
	\Phi(r_{j+1})\leq \tau \Phi(r_j).
$$
	
	To this end,  we put 
$$
	\Sigma_1:=\left\{i\in \mathbb{N}: \Phi(\rho_{i+1})\leq \tau\Phi(\rho_{i}) \right\},
$$
and	
$$
	\Sigma_2:=\left\{i\in \mathbb{N}: \Phi(\rho_i)\geq \varepsilon M(\rho_{i}/2),~~M(\rho_{i}/4)\leq \delta M(\rho_{i}/2) \right\}.
$$
	
	Thanks to Lemma \ref{leau}, we get $\Sigma_1\cup\Sigma_2=\mathbb{N}.$ We now consider the following two cases.
	
	\noindent \underline{Case 1:} $[jh,(j+1)h]\cap \Sigma_2=\{n_1,...,n_q\}$ contains more than $k$ points. Then, 
	\begin{align*}
	\Phi(r_{j+1})&\leq M(\rho_{(j+1)h})\leq M(\rho_{n_q}/4)\leq \delta M(\rho_{n_q}/2) \leq \delta M(\rho_{n_{q-1}}/4) \\
	& \leq \dots \leq \delta^k M(\rho_{n_1}/2)\leq \delta^k\varepsilon^{-1}\Phi(\rho_{n_1})\leq \delta^k\varepsilon^{-1}\Phi(r_j).
	\end{align*} 
	
	Thus  we have $ 	\Phi(r_{j+1})\leq \tau \Phi(r_j).$

	\noindent \underline{Case 2:}  $[jh,(j+1)h]\cap \Sigma_2$ contains less than $k$ points. Then $[jh,(j+1)h]\cap \Sigma_1$ contains a maximal string of
	consecutive integers $n_0,n_0+1,...,n_0+m$ which has more  than $h/k-1$ numbers.  Moreover, by maximality we have   $n_0-1$ and $n_0+m+1$ belong to $\Sigma_2$. Thus 
	\begin{align}\label{st1}
	\Phi(\rho_{n_0+m+1})=\Phi(\tau \rho_{n_0+m})\leq \tau^{\frac{h}{k}-1}\Phi(\rho_{n_0}).
	\end{align} 
	
	To estimate $\Phi(\rho_{n_0+m+1})$ from below,   we consider the following possibilities:
	
	i) If $n_0+m+1=(j+1)h$, then $\Phi(r_{j+1})=\Phi(\rho_{n_0+m+1})$.
	
	ii) If $n_0+m+1<(j+1)h$, then 
$$
	\Phi(r_{j+1}) \leq M(\rho_{(j+1)h})\leq M(\rho_{n_0+m+1}/2) \leq \varepsilon^{-1}\Phi(\rho_{n_0+m+1}).
$$
	
	Thus in both cases we have 
	\begin{align}\label{i+ii}
	\Phi(r_{j+1}) \leq  \varepsilon^{-1}\Phi(\rho_{n_0+m+1}).
	\end{align}
	
	On the other hand, to estimate $\Phi(\rho_{n_0})$ from above,   we consider the following possibilities:
	
	a) If $n_0=jh$, then $\Phi(\rho_{n_0)}=\Phi(r_j)$.
	
	b) If $n_0>jh$, then $n_0-1\in [jh, (j+1)h]\cap \Sigma_2$. In this case, we let $m_0$ be the smallest integer in $[jh,(j+1)h]\cap \Sigma_2\cap(-\infty,n_0-1]$. Then we have  
$$
	\Phi(\rho_{n_0})\leq M(\rho_{n_0-1}/2)\leq M(\rho_{m_0}/2)\leq \varepsilon^{-1}\Phi(\rho_{m_0}).
$$
	
	Since either $m_0=jh$ or $jh,...,m_0-1\in \Sigma_1$, we then find
$$
	\Phi(\rho_{n_0})\leq  \varepsilon^{-1}\tau^{m_0-jh}\Phi(\rho_{jh})\leq \varepsilon^{-1}\Phi(r_{j}).
$$
	 
%
%
%

	Thus in both cases we have
	\begin{align}\label{abc}
	\Phi(\rho_{n_0})\leq \varepsilon^{-1}\Phi(r_{j}).
	\end{align}
	
	Finally, combining \eqref{st1}, \eqref{i+ii} and \eqref{abc} we  find that  $$\Phi(r_{j+1})\leq \tau^{\frac{h}{k}-1}\varepsilon^{-2}\Phi(r_j) \leq \tau \Phi(r_j),$$
	which completes the proof of the theorem.
\end{proof}\\
\begin{lemma}\label{RPhi} Under  	\eqref{A01} and \eqref{A02}, there exsit $C>0$ and $\theta\in(0,1)$ such that for any $B_R(x_0)\subset\Omega$ we have 	$$\fint_{B_{R/2}(x_0)} |V(\nabla v) - V(z_0)|^{2} \leq C \left(\fint_{B_{R}(x_0)} |V(\nabla v) - V(z_0)|^{2\theta}\right)^{\frac{1}{\theta}},$$
	for any vector $z_0\in \mathbb{R}^n$.
\end{lemma}

\begin{proof} Note that \eqref{Aandxi} and \eqref{Aonly} can be equivalently written as 
	\begin{equation}\label{Aex1}	
	(A_0(\xi)-A_0(\eta))\cdot(\xi-\eta)\simeq (s^2 + |\eta|^2 +|\xi-\eta|^2)^{\frac{p-2}{2}}|\xi-\eta|^2
	\end{equation}
	and 
	\begin{equation*}
	|A_0(\xi)-A_0(\eta)| \simeq (s^2 + |\eta|^2 +|\xi-\eta|^2)^{\frac{p-2}{2}}|\xi-\eta|.
	\end{equation*}

	Also, by \eqref{ieq2} we find
	\begin{equation}\label{Viex}
	|V(\xi) -V(\eta)|^2 \simeq (s^2 + |\eta|^2 +|\xi-\eta|^2)^{\frac{p-2}{2}}|\xi-\eta|^2.
	\end{equation}
	
	Let $\varphi: [0, \infty)\rightarrow [0, \infty)$  be the $N$-function defined by 
	\begin{equation}\label{thephi}
	\varphi(t):=\int_{0}^t (s^2 + u^2)^\frac{p-2}{2} u du\simeq (s^2 + t^2)^\frac{p-2}{2} t^2.
	\end{equation}

	Then the  complementary function $\varphi^*$ of $\varphi$ is given by
	\begin{equation} \label{varphi*def}
	\varphi^*(u)=\sup_{t\geq 0}(ut -\varphi(t))=\int_{0}^u (\varphi')^{-1}(t) dt,
	\end{equation}
	where $(\varphi')^{-1}(t)$ is the inverse function of $\varphi'(u)=(s^2 + u^2)^\frac{p-2}{2} u$.
	
	By  noticing that $s^2+t^2 \simeq t^2$ when $s\leq t$ and $s^2+t^2 \simeq s^2$ when $s\geq t$,
	it is easy to see that 
	$$(\varphi')^{-1}(t) \simeq (s^{2(p-1)} + t^2)^{\frac{p'-2}{2}}t$$ 
	uniformly in  $t\geq 0$. Thus it follows from \eqref{varphi*def} that 
	$$\varphi^*(u) \simeq (s^{2(p-1)} + u^2)^{\frac{p'-2}{2}}u^2, \qquad p'=\frac{p}{p-1}.$$
	
	We remark that  
	both $\varphi$ and $\varphi^*$ satisfy the $\Delta_2$-condition, i.e., $\varphi(2t)\leq c \varphi(t)$ and $\varphi^*(2t)\leq c \varphi^*(t)$ for all $t\geq 0$. Here the constant $c$ is independent of $s$, $t$, and $a$.

	The bounds \eqref{Aex1}--\eqref{Viex} enable us to follow the argument in the proof of \cite[Lemma 3.4]{DSV}, using the $N$-function $\varphi$  
	defined in \eqref{thephi} to complete the proof of the lemma.
\end{proof}

We are now ready to prove Theorem \ref{hol-Du}.

\noindent \begin{proof}[Proof of Theorem \ref{hol-Du}] 
	For any $z_0\in \mathbb{R}^n$, using  \eqref{ieq2} we have
	$$|U_q(\nabla v)-U_q(z_0)| \simeq h_{|z_0|}(|\nabla v-z_0|),$$
	where 
	$$h_{|z_0|}(t)=(s^2 + |z_0|^2 + t^2)^{\frac{q-2}{2}} t.$$

	We now let 
	$$g_{|z_0|^{q-1}}(t)= (s^{2(q-1)} + |z_0|^{2(q-1)} + t^2)^{\frac{p}{2(q-1)}-1}\, t^2.$$

	Then we have
	$$g_{|z_0|^{q-1}}(h_{|z_0|}(t)) \simeq (s^2 + |z_0|^2 +t^2)^{\frac{p-2}{2}}t^2,$$
	and thus by  \eqref{Viex}  it holds that   	 
	\begin{equation}\label{V-Uq}
	|V(\nabla v)-V(z_0)|^2 \simeq  g_{|z_0|^{q-1}}(|U_q(\nabla v)-U_q(z_0)|).
	\end{equation}

	Let $R_m:=2^{-m}(R/2)$ for $m\in \mathbb{Z}$. To prove \eqref{HU}, it is enough to show it with $\rho= R_m$ for all sufficiently large $m\in \mathbb{N}$.

	By Theorem \ref{calphaforphi}	there exists $\sigma_1 \in (0,1)$ such that 
	\begin{align*}
	\fint_{B_{R_m}} |V(\nabla v) - [V(\nabla v)]_{B_{R_m}}|^2 & \leq C 2^{-2m\sigma_1} \fint_{B_{R/2}} |V(\nabla v) - [V(\nabla v)]_{B_{R/2}}|^2\\
	&\leq C 2^{-2m\sigma_1} \fint_{B_{R/2}} |V(\nabla v) - V(z_0)|^2,
	\end{align*}	
	where  $z_0$ is chosen so that $ U_q(z_0) = [U_q(\nabla v)]_{B_{R}}$.	Thus it follows from \eqref{V-Uq}, Lemma \ref{RPhi}, and \cite[Corollary 3.4]{Diening} that 
	\begin{align}\label{v-z0}
	\fint_{B_{R_m}} & |V(\nabla v) - [V(\nabla v)]_{B_{R_m}}|^2\nonumber\\
	 & \leq C 2^{-2m\sigma_1} g_{|z_0|^{q-1}}\left(\fint_{B_{R}} |U_q(\nabla v)-U_q(z_0)|\right).
	\end{align}

	Note that $s^{2(q-1)} + |z_0|^{2(q-1)} \simeq s^{2(q-1)} + |U_q(z_0)|^2$ and thus  by \cite[Corollary 26]{DK}, for any $z\in\mathbb{R}^n$, we have  
	\begin{align*}
	g_{|z_0|^{q-1}}(t)&\simeq (s^{2(q-1)} + |U_q(z_0)|^2 + t^2)^{\frac{p}{2(q-1)}-1}\, t^2\\
	&\leq  C (s^{2(q-1)} + |U_q(z)|^2 + t^2)^{\frac{p}{2(q-1)}-1}\, t^2 \\
	& \quad +  C (s^{2(q-1)} + |U_q(z_0)|^2 +|U_q(z)|^2)^{\frac{\frac{p}{q-1}-2}{2}}|U_q(z_0)-U_q(z)|^2\\
	&\leq  C (s^{2(q-1)} + |z|^{2(q-1)} + t^2)^{\frac{p}{2(q-1)}-1}\, t^2 \\
	& \quad +  C (s^{2(q-1)} + |z_0|^{2(q-1)} +|z|^{2(q-1)})^{\frac{\frac{p}{q-1}-2}{2}}|U_q(z_0)-U_q(z)|^2.
	\end{align*}
	
	Then using  \eqref{ieq2} we get 	
	\begin{align}
	g_{|z_0|^{q-1}}(t)&\leq  C (s^{2(q-1)} + |z|^{2(q-1)} + t^2)^{\frac{p}{2(q-1)}-1}\, t^2 \nonumber \\
	& \quad +  C(s^{2} + |z_0|^{2} +|z|^{2})^{\frac{p-2(q-1)}{2}} (s^2 + |z_0|^2 + |z|^2)^{q-2} |z_0-z|^2 \nonumber\\
	&\leq  C (s^{2(q-1)} + |z|^{2(q-1)} + t^2)^{\frac{p}{2(q-1)}-1}\, t^2  +  C (s^{2} + |z_0|^{2} +|z|^{2})^{\frac{p-2}{2}}  |z_0-z|^2 \nonumber\\
	&\leq  C g_{|z|^{q-1}}(t) + C |V(z_0)-V(z)|^2. \label{shift}
	\end{align}
	
	We now let $\xi_m\in \mathbb{R}^n$ be such that $U_q(\xi_m)=[U_q(\nabla v)]_{B_{R_m}}$.  Then applying \eqref{shift} with $z=\xi_m$ and \eqref{V-Uq} we find
	\begin{align*}
	&g_{|z_0|^{q-1}}\left(\fint_{B_{R}} |U_q(\nabla v)-U_q(z_0)|\right)\\
	&\leq C g_{|\xi_m|^{q-1}}\left(\fint_{B_{R}} |U_q(\nabla v)-U_q(z_0)|\right) + C |V(z_0)-V(\xi_m)|^2\\
	&\leq C g_{|\xi_m|^{q-1}}\left(\fint_{B_{R}} |U_q(\nabla v)-U_q(z_0)|\right) + C g_{|\xi_m|^{q-1}}(|U_q(z_0)-U_q(\xi_m)|) \\
	&\leq C g_{|\xi_m|^{q-1}}\left(\fint_{B_{R}} |U_q(\nabla v)-[U_q(\nabla v)]_{B_{R}}|\right) \\
	&\quad + C g_{|\xi_m|^{q-1}}(|[U_q(\nabla v)]_{B_{R}}-[U_q(\nabla v)]_{B_{R_m}}|).
	\end{align*}

	We next observe that 
	\begin{align*}
	|[U_q(\nabla v)]_{B_{R}}- [U_q(\nabla v)]_{B_{R_m}}|& \leq \sum_{k=-1}^{m-1} |[U_q(\nabla v)]_{B_{R_{k+1}}}- [U_q(\nabla v)]_{B_{R_{k}}}|\\
	&\leq \sum_{k=-1}^{m-1} \fint_{B_{R_{k+1}}} |U_q(\nabla v)- [U_q(\nabla v)]_{B_{R_{k}}}|\\
	&\leq  2^n \sum_{k=-1}^{m-1} \fint_{B_{R_{k}}} |U_q(\nabla v)- [U_q(\nabla v)]_{B_{R_{k}}}|.
	\end{align*}
	
	Thus by the monotonicity  of $g_{|\xi_m|^{q-1}}$  we get 
	\begin{align*}
	g_{|z_0|^{q-1}} &\left(\fint_{B_{R}} |U_q(\nabla v)-U_q(z_0)|\right) \\
	& \leq  C g_{|\xi_m|^{q-1}}\left(2^n\sum_{k=-1}^{m-1} \fint_{B_{R_{k}}} |U_q(\nabla v)- [U_q(\nabla v)]_{B_{R_{k}}}|\right).
	\end{align*}
	
	Now in view of \eqref{v-z0}, this yields
	\begin{align}
	\fint_{B_{R_m}} & |V(\nabla v) - [V(\nabla v)]_{B_{R_m}}|^2 \nonumber \\
	& \leq C 2^{-2m\sigma_1}  g_{|\xi_m|^{q-1}}\left(2^n\sum_{k=-1}^{m-1} \fint_{B_{R_{k}}} |U_q(\nabla v)- [U_q(\nabla v)]_{B_{R_{k}}}|\right).\label{Vandsum}
	\end{align}

	Let $\eta_m$ be such that $V(\eta_m)=[V(\nabla v)]_{B_{R_m}}$. Then by \eqref{V-Uq} we have 
	$$\fint_{B_{R_m}} g_{|\eta_m|^{q-1}}(|U_q(\nabla v) -U_q(\eta_m)|) \leq C \fint_{B_{R_m}}  |V(\nabla v) - [V(\nabla v)]_{B_{R_m}}|^2,$$
	which by Jensen's inequality  and the monotonicity of $g_{|\eta_m|^{q-1}}$  gives 
	\begin{equation}\label{gout}
	g_{|\eta_m|^{q-1}}\left( \frac 1 2 \fint_{B_{R_m}} |U_q(\nabla v) -[U_q(\nabla v)]_{B_{R_m}}| \right) \leq C \fint_{B_{R_m}}  |V(\nabla v) - [V(\nabla v)]_{B_{R_m}}|^2
	\end{equation}
	Combining \eqref{Vandsum} and \eqref{gout} we get 
	\begin{align}
	g_{|\eta_m|^{q-1}} &\left( \frac 1 2 \fint_{B_{R_m}}  |U_q(\nabla v) -[U_q(\nabla v)]_{B_{R_m}}| \right) \nonumber \\
	& \leq C 2^{-2m\sigma_1}  g_{|\xi_m|^{q-1}}\left(2^n\sum_{k=-1}^{m-1} \fint_{B_{R_{k}}} |U_q(\nabla v)- [U_q(\nabla v)]_{B_{R_{k}}}|\right).\label{gCg}
	\end{align}	
	Note that for any $\lambda \in (0,1)$ we have $$g_{|\xi_m|^{q-1}}(\lambda t) \geq \lambda^{\kappa} \, g_{|\xi_m|^{q-1}}( t), \quad {\rm where} \, \kappa=\max\left \{\frac{p}{q-1}, 2\right \}.$$  
	Thus \eqref{gCg} yields that 
	\begin{align}
	g_{|\eta_m|^{q-1}} &\left( \frac 1 2 \fint_{B_{R_m}}  |U_q(\nabla v) -[U_q(\nabla v)]_{B_{R_m}}| \right) \nonumber \\
	& \leq  g_{|\xi_m|^{q-1}}\left(2^n C^{\frac 1 \kappa} 2^{\frac{-2m\sigma_1}{\kappa}}  \sum_{k=-1}^{m-1} \fint_{B_{R_{k}}} |U_q(\nabla v)- [U_q(\nabla v)]_{B_{R_{k}}}|\right),\label{Vandsum2}
	\end{align}	
	provided $C 2^{-2m\sigma_1}<1$, i.e., provided $m$ is sufficiently large.\\
	We now apply the inverse function of $g_{|\eta_m|^{q-1}}$ to both sides of \eqref{Vandsum2} to arrive at 
	\begin{align*}
	\fint_{B_{R_m}}  |U_q(\nabla v) -[U_q(\nabla v)]_{B_{R_m}}| &  \leq  C 2^{\frac{-2m\sigma_1}{\kappa}}  \sum_{k=-1}^{m-1} \fint_{B_{R_{k}}} |U_q(\nabla v)- [U_q(\nabla v)]_{B_{R_{k}}}|\\
	& \leq C 2^{\frac{-2m\sigma_1}{\kappa}} (m+1)   \max_{-1\leq k<m} \fint_{B_{R_{k}}} |U_q(\nabla v)- [U_q(\nabla v)]_{B_{R_{k}}}|
	\end{align*}	
	for all sufficiently large $m$.\\
	Let $0<\alpha< \frac{2\sigma_1}{\kappa}$. From the above inequality we have 
	\begin{align*}
	R_m^{-\alpha} \fint_{B_{R_m}} & |U_q(\nabla v) -[U_q(\nabla v)]_{B_{R_m}}| \\
	& \leq C 2^{\frac{-2m\sigma_1}{\kappa}} (m+1)   \max_{-1\leq k<m}  R_m^{-\alpha} R_k^{\alpha}  R_k^{-\alpha}  \fint_{B_{R_{k}}} |U_q(\nabla v)- [U_q(\nabla v)]_{B_{R_{k}}}|\\
	& \leq C 2^{\frac{-2m\sigma_1}{\kappa} + m \alpha} (m+1)   \max_{-1\leq k<m}   R_k^{-\alpha}  \fint_{B_{R_{k}}} |U_q(\nabla v)- [U_q(\nabla v)]_{B_{R_{k}}}|\\
	& \leq \frac 1 2    \max_{-1\leq k<m}   R_k^{-\alpha}  \fint_{B_{R_{k}}} |U_q(\nabla v)- [U_q(\nabla v)]_{B_{R_{k}}}|,
	\end{align*}	
	provided $m\geq m_0$ where $m_0$ sufficiently large so that we have both   $C 2^{-2m_0\sigma_1}<1$ and $C 2^{\frac{-2m_0\sigma_1}{\kappa} + m_0 \alpha} (m_0+1)<\frac 1 2 $.
	This is possible since $\alpha<\frac{2\sigma_1}{\kappa}$.
	
	For any $\ell=2, 3, \dots$, we now apply the previous inequality with $m_0\leq m\leq \ell m_0$ to deduce that 
	\begin{align*}
	\max_{m_0\leq m\leq \ell m_0} R_m^{-\alpha} \fint_{B_{R_m}} & |U_q(\nabla v) -[U_q(\nabla v)]_{B_{R_m}}| \\
	& \leq \frac 1 2    \max_{-1\leq k<\ell m_0}   R_k^{-\alpha}  \fint_{B_{R_{k}}} |U_q(\nabla v)- [U_q(\nabla v)]_{B_{R_{k}}}|\\
	& \leq \frac 1 2    \max_{-1\leq k<m_0}   R_k^{-\alpha}  \fint_{B_{R_{k}}} |U_q(\nabla v)- [U_q(\nabla v)]_{B_{R_{k}}}|\\
	& +  \frac 1 2 \max_{m_0 \leq k<\ell m_0}   R_k^{-\alpha}  \fint_{B_{R_{k}}} |U_q(\nabla v)- [U_q(\nabla v)]_{B_{R_{k}}}|.
	\end{align*}	
	This gives 
	\begin{align*}
	\max_{m_0\leq m\leq \ell m_0} R_m^{-\alpha} \fint_{B_{R_m}} & |U_q(\nabla v) -[U_q(\nabla v)]_{B_{R_m}}| \\
	&\leq     \max_{-1\leq k<m_0}   R_k^{-\alpha}  \fint_{B_{R_{k}}} |U_q(\nabla v)- [U_q(\nabla v)]_{B_{R_{k}}}|\\
	& \leq C R^{-\alpha} \fint_{B_{R}} |U_q(\nabla v)- [U_q(\nabla v)]_{B_{R}}|,
	\end{align*}	
	which  completes the proof of Theorem \ref{hol-Du}.
\end{proof}	

 \section{Interior pointwise gradient estimates}  

The main goal of this section is to prove Theorem \ref{grad1}. We shall need some preliminary results for that purpose. 

Let $u\in W_{\rm loc}^{1,p}(\Omega)$ be a solution of \eqref{quasi-measure} and suppose that   $B_{2r}=B_{2r}(x_0)\subset\subset\Omega$.  We consider the unique solution $w\in  u+ W_{0}^{1,p}(B_{2r})$
to the  equation 
\begin{equation}
\label{111120146}\left\{ \begin{array}{rcl}
- \operatorname{div}\left( {A(x,\nabla w)} \right) &=& 0 \quad {\rm in} \quad B_{2r}, \\ 
w &=& u \quad {\rm on} \quad \partial B_{2r}.  
\end{array} \right.
\end{equation}

We first recall the following  version of interior Gehring's lemma that can be found  in \cite[Theorem 6.7]{Giu}.
\begin{lemma} Let $w$ be as in \eqref{111120146}.
	There exist  constants $\theta_1>p$ and $C>0$ depending only on $n,\Lambda$ such that the  estimate      
	\begin{equation}\label{111120148}
		\left(\fint_{B_{\rho/2}(y)}(|\nabla w|+s)^{\theta_1} dxdt\right)^{\frac{1}{\theta_1}}\leq C\left(\fint_{B_{\rho}(y)}(|\nabla w|+s)^{t} dx\right)^{\frac{1}{t}},
	\end{equation}
	holds for all  $B_{\rho}(y)\subset B_{2r}(x_0)$ and $t>0$.
\end{lemma} 
The following important  comparison  estimate can be found  in \cite[Lemma 2.2]{QH4}.
    \begin{lemma}\label{111120149}Let $w$ be as in \eqref{111120146} and  assume that $\frac{3n-2}{2n-1}<p\leq  2-\frac{1}{n}$.  Then it holds that  for any $\gamma_0\in\left(\frac{n}{2n-1},\frac{(p-1)n}{n-1}\right)$,
    	\begin{align*}
   & \left(	\fint_{B_{2r}}|\nabla u-\nabla w|^{\gamma_0}dx\right)^{\frac{1}{\gamma_0}}\\
   & \quad \leq C \left[\frac{|\mu|(B_{2r})}{r^{n-1}}\right]^{\frac{1}{p-1}}+C\frac{|\mu|(B_{2r})}{r^{n-1}}\left(	\fint_{B_{2r}}(|\nabla u|+s)^{\gamma_0}dx\right)^{\frac{2-p}{\gamma_0}}.\label{1111201410+}
    	\end{align*}
    	where $C$ is a constant only depending on $n,p,\Lambda,\gamma_0$.
\end{lemma}

We remark that the range of $\gamma_0$ was not explicitly stated in \cite[Lemma 2.2]{QH4} but it can be easily seen from the proof of \cite[Lemma 2.2]{QH4}. Moreover, only the case $s=0$
was considered in \cite[Lemma 2.2]{QH4}, but the proof works also  in the case  $s>0$.


 We now let $v\in W_0^{1,p}(B_r(x_0))$ be the  unique solution of 
\begin{equation*}
\left\{ \begin{array}{rcl}
- \operatorname{div}\left( {A(x_0,\nabla v)} \right) &=& 0 \quad {\rm in} \quad B_{r}, \\ 
v &=& w \quad {\rm on} \quad \partial B_{r}.  
\end{array} \right.
\end{equation*}

 By standard regularity, we have for any $t>0$
\begin{align}\label{es29}
||\nabla v||_{L^\infty(B_{r/2})}\leq C \left(\fint_{B_r}|\nabla v|^{t}\right)^{1/t}.
\end{align}

We also have an estimate for the difference  $\nabla v-\nabla w$,
$$
\fint_{B_{r}}|\nabla v-\nabla w|^p dx \leq C \omega(r)^p \fint_{B_{r}}(|\nabla w| +s)^p dx.
$$
The  proof of this fact can be found in \cite[Equ. (4.35)]{Duzamin2}. Thus by \eqref{111120148} and H\"older's inequality, we get 
\begin{align}\label{es23}
\fint_{B_{r}}|\nabla v-\nabla w|^{\gamma_0} dx \leq C \omega(r)^{\gamma_0} \fint_{B_{2r}}(|\nabla w| +s)^{\gamma_0} dx.
\end{align}


 For a ball $B_\rho=B_\rho(x_0)\subset\Omega$, we  now define
$$
\mathbf{I}(\rho)=\mathbf{I}(x_0,\rho):=	\fint_{B_{\rho}}|U_{\gamma_0+1}(\nabla u)-\left[U_{\gamma_0+1}(\nabla u)\right]_{B_{\rho}}|dx.
$$
\begin{proposition}   Suppose  that $u\in W_{\rm loc}^{1,p}(\Omega)$ is a solution of \eqref{quasi-measure}. Then there exists $\alpha_0\in (0,1/2)$ such that 
 for any $\varepsilon\in (0,1)$ and $B_{2r}(x_0)\Subset\Omega$  we have 
	\begin{align}\nonumber
	\mathbf{I}(\varepsilon r)&\leq C\varepsilon^{\alpha_0}\mathbf{I}(r)+C_\varepsilon\left(\frac{|\mu|(B_{2r})}{r^{n-1}}\right)^{\frac{\gamma_0}{p-1}}\\ & ~~ +C_\varepsilon\left(\frac{|\mu|(B_{2r})}{r^{n-1}}\right)^{\gamma_0}\left(	\fint_{B_{2r}}(|\nabla u|+s)^{\gamma_0}\right)^{2-p} +	C_\epsilon\,  \omega(r)^{\gamma_0}\fint_{B_{2r}} (|\nabla u|+s)^{\gamma_0}\label{intA},
	\end{align}
	where $C_\varepsilon$ is a constant depending on $\varepsilon,n,p,\Lambda,\alpha$.
\end{proposition}


\begin{proof}
Since 	$\gamma_0\leq  1$, using \eqref{ieq2} we have 
$$|U_{\gamma_0+1}(\nabla u)-U_{\gamma_0+1}(\nabla v)|\leq C |\nabla u-\nabla v|^{\gamma_0}.$$

Thus by Theorem \ref{hol-Du}, we can find $\alpha_0\in (0,1/2)$ such that 
	\begin{align}\nonumber
&\fint_{B_{\varepsilon r}}|U_{\gamma_0+1}(\nabla u)-\left[U_{\gamma_0+1}(\nabla u)\right]_{B_{\varepsilon r}}|\\~~~\nonumber&\leq
C\fint_{B_{\varepsilon r}}|U_{\gamma_0+1}(\nabla v)-\left[U_{\gamma_0+1}(\nabla v)\right]_{B_{\varepsilon r}}|+C	\fint_{B_{\varepsilon r}} |\nabla u-\nabla v|^{\gamma_0}
		\\&\leq
	C	\varepsilon^{\alpha_0}\fint_{B_{r}}|U_{\gamma_0+1}(\nabla v)-\left[U_{\gamma_0+1}(\nabla v)\right]_{B_{r}}|+C\varepsilon^{-n}	\fint_{B_{r}} |\nabla u-\nabla v|^{\gamma_0}\nonumber
	\\&\leq
	C\varepsilon^{\alpha_0}\fint_{B_{r}}|U_{\gamma_0+1}(\nabla u)-\left[U_{\gamma_0+1}(\nabla u)\right]_{B_{r}}|+C\varepsilon^{-n}	\fint_{B_{r}} |\nabla u-\nabla v|^{\gamma_0}.\label{es24}
	\end{align}

Moreover,  by  \eqref{es23} and the fact that $|\omega(r)|\leq 1$, one has 
	\begin{align}\nonumber
	\fint_{B_{r}} |\nabla u-\nabla v|^{\gamma_0}
&\leq 	C\fint_{B_{r}} |\nabla u-\nabla w|^{\gamma_0}+	C\fint_{B_{r}} |\nabla w- \nabla v|^{\gamma_0}
	\\&\nonumber~~~\leq 	C\fint_{B_{2r}} |\nabla u-\nabla w|^{\gamma_0}+	C \omega(r)^{\gamma_0}\fint_{B_{2r}} (|\nabla w|+s)^{\gamma_0}
		\\&~~~\leq 	C\fint_{B_{2r}} |\nabla u- \nabla w|^{\gamma_0}+	C \omega(r)^{\gamma_0}\fint_{B_{2r}} (|\nabla u|+s)^{\gamma_0}\label{es25}.
	\end{align}

We then derive  from \eqref{es24} and \eqref{es25}  that 
	\begin{align}
\mathbf{I}(\varepsilon r)\leq C\varepsilon^{\alpha_0}\mathbf{I}(r)+C_\varepsilon \fint_{B_{2r}} |\nabla u-\nabla w|^{\gamma_0}+	C_\varepsilon \omega(r)^{\gamma_0}\fint_{B_{2r}} (|\nabla u|+s)^{\gamma_0}.\label{es27}
\end{align}

At this point we apply  Lemma \ref{111120149} to bound the second term on the right-hand side of   \eqref{es27}. This yields \eqref{intA} as desired.
\end{proof}

We are now we are ready to prove Theorem \ref{grad1}.

\noindent \begin{proof}[Proof of Theorem \ref{grad1}] We shall prove \eqref{Wpointwise} at $x=x_0$ and $B_R(x_0)\subset\Omega$.
	  Let  $U(x):= U_{\gamma_0+1}(\nabla u(x))$ and choose $\varepsilon <1/4$ small enough so that $
	C\varepsilon ^{\alpha_0}\leq \frac{1}{4}$, where $C$ is the constant in \eqref{intA}.

	Set $R_j=\varepsilon^{j} R$, $B_j:=B_{2R_j}(x_0), \mathbf{I}_j=\mathbf{I}(R_j)$ and $T_j:=\fint_{B_{j}} (|\nabla u| +s)^{\gamma_0} dx$. Applying \eqref{intA}  yields 
$$
	\mathbf{I}_{j+1}\leq \frac{1}{4}\mathbf{I}_{j}+C\left( \frac{|\mu|(B_{j})}{R_{j}^{n-1}}\right)^{\frac{\gamma_0}{p-1}} +C\left(\frac{|\mu|(B_{j})}{R_{j}^{n-1}}\right)^{\gamma_0} T_j^{2-p}+	C \omega(R_j)^{\gamma_0}T_j.
$$
	
	Summing this up over $j\in \{j_0,j_0+1,2,...,m-1\}$, we obtain 
	\begin{align}\label{z1}
	\sum_{j=j_0}^{m}\mathbf{I}_{j}&\leq C\, \mathbf{I}_{j_0}+C\sum_{j=j_0}^{m-1}\left(\frac{|\mu|(B_{j})}{R_{j}^{n-1}}\right)^{\frac{\gamma_0}{p-1}} \nonumber\\
	&\quad  +C\sum_{j=j_0}^{m-1}\left(\frac{|\mu|(B_{j})}{R_{j}^{n-1}}\right)^{\gamma_0}T_j^{2-p}+	C\sum_{j=j_0}^{m-1} \omega(R_j)^{\gamma_0} T_j.
	\end{align}
	
	Since $$
	\sum_{j=j_0}^{m}\mathbf{I}_{j}\geq C \sum_{j=j_0}^{m}|\left[U\right]_{B_{j+1}}-\left[U\right]_{B_{j}}|\geq C |\left[U\right]_{B_{m+1}}-\left[U\right]_{B_{j_0}}|,$$
	we see that  \eqref{z1} implies 
		\begin{align}\nonumber
|\left[U\right]_{B_{m+1}}| & +\sum_{j=j_0}^{m}\mathbf{I}_{j} \leq  C\, \mathbf{I}_{j_0}+|\left[U\right]_{B_{j_0}}|+C\sum_{j=j_0}^{m-1}\left(\frac{|\mu|(B_{j})}{R_{j}^{n-1}}\right)^{\frac{\gamma_0}{p-1}}\\&+C\sum_{j=j_0}^{m-1}\left(\frac{|\mu|(B_{j})}{R_{j}^{n-1}}\right)^{\gamma_0}T_j^{2-p}+	C\sum_{j=j_0}^{m-1}\omega(R_j)^{\gamma_0} T_j.\label{es1}
	\end{align}

By \eqref{dini}, there is $j_0=j_0(\varepsilon,C,D)>1$   large enough such that 
	\begin{align}\label{co3}
	\varepsilon^{-n}C\sum_{j=j_0}^{\infty}\omega(R_j)^{\gamma_0}\leq \frac{1}{10},
	\end{align}
	where $C$ is the constant in \eqref{es1}. 	
	
	Note that 
	\begin{align}\label{es0}
	\sum_{j=j_0}^{m}\left(\frac{|\mu|(B_{j})}{R_{j}^{n-1}}\right)^{\gamma_0}\leq C\, \int_{0}^{2R_{j_0-1}}\left(\frac{|\mu|(B_{\rho}(x_0))}{\rho^{n-1}}\right)^{\gamma_0}\frac{d\rho}{\rho},
	\end{align} 
and since $p<2$ we also have 
\begin{align}\label{es00}
\sum_{j=j_0}^{m}\left(\frac{|\mu|(B_{j})}{ R_{j}^{n-1}}\right)^{\frac{\gamma_0}{p-1}}\leq C \left(\int_{0}^{2 R_{j_0-1}}\left(\frac{|\mu|(B_{\rho}(x_0))}{\rho^{n-1}}\right)^{\gamma_0}\frac{d\rho}{\rho}\right)^{\frac{1}{p-1}}.
\end{align} 	
	
	Moreover, since $\gamma_0\leq 1$ we have 	$ |U|\leq |\nabla u|^{\gamma_0} $, and thus to prove \eqref{Wpointwise} at $x=x_0$
	  it is enough to show that 
	\begin{align}\label{maines}
	|U(x_0)|\leq C T_{j_0}+ C   \left(\int_{0}^{2 R_{j_0-1}}\left(\frac{|\mu|(B_{\rho}(x_0))}{\rho^{n-1}}\right)^{\gamma_0}\frac{d\rho}{\rho}\right)^{\frac{1}{p-1}}.
	\end{align}
	
	To prove \eqref{maines} we consider the following possibilities:

\medskip
	
\noindent 	\underline{Case 1:}  If $
	|U(x_0)|\leq  T_{j_0}$,  then  \eqref{maines} trivially follows.
	
\medskip
	\noindent 	\underline{Case 2:} If 
		\begin{equation}\label{co1}
	T_j\leq |U(x_0)|~~\forall j_0\leq j\leq j_1~~\text{and }  |U(x_0)|<T_{j_1+1},
	\end{equation}
	then since $\gamma_0\leq 1$ we have 
	\begin{align*}
 |U(x_0)|& <	\fint_{B_{j_1 +1}} (|\nabla u| +s)^{\gamma_0}dx \\
 &\leq \fint_{B_{j_1 +1}} |\nabla u|^{\gamma_0} dx + s^{\gamma_0}
 \leq \mathbf{I}_{j_1+ 1}+ |\left[U\right]_{B_{j_1 +1}}| + s^{\gamma_0}\\
 & \leq  \varepsilon^{-n}\mathbf{I}_{j_1}+ |\left[U\right]_{B_{j_1 +1}}| + s^{\gamma_0}.
	\end{align*}

Now applying \eqref{es1} with  $m=j_1$ and using  \eqref{es0}, \eqref{es00} and \eqref{co1}
we get 
		\begin{align*}
	|U(x_0)|&< C_\epsilon\, \mathbf{I}_{j_0}+C_\epsilon\, |\left[U\right]_{B_{j_0}}|+C_\epsilon\, \left[\int_{0}^{2 R_{j_0-1}}\left(\frac{|\mu|(B_{\rho}(x_0))}{\rho^{n-1}}\right)^{\gamma_0}\frac{d\rho}{\rho}\right]^{\frac{1}{p-1}}\\&~~+C_\epsilon  \left[\int_{0}^{2 R_{j_0-1}}\left(\frac{|\mu|(B_{\rho}(x_0))}{\rho^{n-1}}\right)^{\gamma_0}\frac{d\rho}{\rho}\right] |U(x_0)|^{2-p}\\&~~+	\varepsilon^{-n}C\sum_{j=j_0}^{m-1} \omega(R_j)^{\gamma_0} |U(x_0)|  + s^{\gamma_0}.
	\end{align*}
	
	Hence using \eqref{co3} and Young's inequality we find 
	\begin{align*}
	|U(x_0)|&  \leq C_\varepsilon\mathbf{I}_{j_0}+C_\varepsilon|\left[U\right]_{B_{j_0}}|+C_\varepsilon \left(\int_{0}^{2 R_{j_0-1}}\left(\frac{|\mu|(B_{\rho}(x_0))}{\rho^{n-1}}\right)^{\gamma_0}\frac{d\rho}{\rho}\right)^{\frac{1}{p-1}}\\
	&~~ +	\frac{1}{5}|U(x_0)| + s^{\gamma_0}.
	\end{align*}
	
	This implies \eqref{maines} as desired.

\medskip
	
	\noindent 	\underline{Case 3:} If $
	T_j\leq |U(x_0)|$ for any $j\geq j_0$, then from \eqref{es1} we have for any $m> j_0$,
		\begin{align*}\nonumber
&|\left[U\right]_{B_{m+1}}|\leq  C\mathbf{I}_{j_0}+|\left[U\right]_{B_{j_0}}|+C\left[\int_{0}^{2 R_{j_0-1}}\left(\frac{|\mu|(B_{\rho}(x_0))}{\rho^{n-1}}\right)^{\gamma_0}\frac{d\rho}{\rho}\right]^{\frac{1}{p-1}}\\&\nonumber~~+C\left[\int_{0}^{2 R_{j_0-1}}\left(\frac{|\mu|(B_{\rho}(x_0))}{\rho^{n-1}}\right)^{\gamma_0}\frac{d\rho}{\rho}\right]|U(x_0)|^{2-p}+	C\sum_{j=j_0}^{m-1} \omega(R_j)^{\gamma_0}|U(x_0)|\\& \leq C\mathbf{I}_{j_0}+|\left[U\right]_{B_{j_0}}|+C\left[\int_{0}^{2 R_{j_0-1}}\left(\frac{|\mu|(B_{\rho}(x_0))}{\rho^{n-1}}\right)^{\gamma_0}\frac{d\rho}{\rho}\right]^{\frac{1}{p-1}}\\&+C\left[\int_{0}^{2 R_{j_0-1}}\left(\frac{|\mu|(B_{\rho}(x_0))}{\rho^{n-1}}\right)^{\gamma_0}\frac{d\rho}{\rho}\right]|U(x_0)|^{2-p}+	\frac{1}{10}|U(x_0)|.
	\end{align*}
	Here we used \eqref{co3} in the last inequality. 
	Letting $m\to\infty$ we get 
	\begin{align*}
&	|U(x_0)|\leq C\mathbf{I}_{j_0}+|[U]_{B_{j_0}}|+C\left[\int_{0}^{2 R_{j_0-1}}\left(\frac{|\mu|(B_{\rho}(x_0))}{\rho^{n-1}}\right)^{\gamma_0}\frac{d\rho}{\rho}\right]^{\frac{1}{p-1}}\\&+C\left[\int_{0}^{2 R_{j_0-1}}\left(\frac{|\mu|(B_{\rho}(x_0))}{\rho^{n-1}}\right)^{\gamma_0}\frac{d\rho}{\rho}\right]|U(x_0)|^{2-p}+	\frac{1}{10}|U(x_0)|.
	\end{align*}

Then using Young's inequality we deduce \eqref{maines}. The proof is complete.
\end{proof}
 
 \section{Global pointwise gradient estimates}
We shall prove Theorem \ref{boundary} in this section. As discussed earlier,  by a standard approximation we may assume that $u\in W_0^{1,p}(\Omega)$ is a solution of 
\eqref{quasi-measure}. We shall prove \eqref{ieq1} for any  $x=x_0\in\Omega$,  a Lebesgue point of $(s^2 +|\nabla u|^2)^{\frac{\gamma_0-1}{2}}\nabla u$. 

By Theorem \ref{grad1} we have 
\begin{align}
|\nabla u(x_0)| &\leq C\, \Big[{\bf P}^{2 {\rm diam}(\Omega)}_{\gamma_0}(|\mu|)(x_0)\Big]^{\frac{1}{\gamma_0 (p-1)}} \nonumber \\
& \qquad  + C\, \Big(\fint_{B_{d(x_0)}(x_0)}|\nabla u(y)|^{\gamma_0} dy\Big)^{\frac{1}{\gamma_0}} + C\, s. \label{Wwithtail}
\end{align}

Recall that by a standard estimate (see, e.g., the proof of \cite[Lemma 2.2]{QH4}), we have 
\begin{equation}\label{int0}
\int_{\Omega}|\nabla u|^{\gamma_0}\leq C\left( {\rm diam}(\Omega)\right)^{n-\frac{\gamma_0(n-1)}{p-1}}|\mu|(\Omega)^{\frac{\gamma_0}{p-1}} + C\,{\rm diam}(\Omega)^{n} s^{\gamma_0}.
\end{equation}

Thus we may assume that $d(x_0)\leq r_1/2$ for any sufficiently small $r_1>0$. Recall that $\Omega$ is a $(\delta,R_0)$-Reifenberg flat domain for some $R_0>0$. 
Therefore, we may further  assume that $ d(x_0)\leq  r_1/2\leq R_0/100 \leq {\rm diam}(\Omega)/1000$.  

Let  $x_1\in \partial\Omega$ be such that $|x_1-x_0|=d(x_0)$. For any $r\in (0, r_1]$
we consider the unique solution $
w\in W_{0}^{1,p}(\Omega_{2r}(x_1))+u$
to the following equation 
\begin{equation}
\label{111120146*}\left\{ \begin{array}{rcl}
- \operatorname{div}\left( {A(x,\nabla w)} \right) &=& 0 \quad {\rm in}~~ \Omega_{2 r}(x_1), \\ 
 w &=& u \quad {\rm on}~~\partial \Omega_{2r}(x_1), 
\end{array} \right.
\end{equation}
where we write  $\Omega_r(x_1)=\Omega\cap B_r(x_1)$.

We have the following boundary counterpart of Lemma \ref{111120149} (see \cite[Lemma 2.5]{QH4}).
\begin{lemma}\label{111120149+} Let $w$ be as in \eqref{111120146*} and $\gamma_0$ be as in Lemma \ref{111120149}. Then it holds that 
	\begin{align*} 
	& \left(	\fint_{B_{2r}(x_1)}  |\nabla u-\nabla w|^{\gamma_0}dx\right)^{\frac{1}{\gamma_0}} \\
	& \quad \leq C \left[\frac{|\mu|(B_{2r}(x_1))}{r^{n-1}}\right]^{\frac{1}{p-1}}  +C\frac{|\mu|(B_{2r}(x_1))}{r^{n-1}}\left(	\fint_{B_{2r}(x_1)} (|\nabla u|+s)^{\gamma_0}dx\right)^{\frac{2-p}{\gamma_0}}.
	\end{align*}
\end{lemma}

Next, we let $v\in w+ W_0^{1,p}(\Omega_{r}(x_1))$ be the unique solution of 
\begin{equation*}
\label{eqfv}\left\{ \begin{array}{rcl}
- \operatorname{div}\left( {A(x_1,\nabla v)} \right) &=& 0 \quad {\rm in} ~~\Omega_{r}(x_1), \\ 
v &=& w\quad  {\rm on}~~\partial \Omega_{r}(x_1). 
\end{array} \right.
\end{equation*}

In what follows, we shall tacitly extend $u$ by zero to $\mathbb{R}^n\setminus \Omega$. Then extend $w$ by $u$ to
$\mathbb{R}^n\setminus \Omega_{2r}(x_1)$ and $v$ by $w$ to
$\mathbb{R}^n\setminus \Omega_{r}(x_1)$.
As in \eqref{es23}, we also have an estimate for the difference $\nabla v- \nabla w:$
 \begin{align}\label{es-1}
\fint_{B_{r}(x_1)}|\nabla v-\nabla w|^{\gamma_0} dx \leq C \omega(r)^{\gamma_0} \fint_{B_{2r}(x_1)}(|\nabla w| +s)^{\gamma_0} dx.
\end{align}

We will need the following boundary counterpart of \eqref{es29}. But here, due to the possible irregularity of $\Omega$, we only have  $L^q$-estimates for the gradient of $v$ for any large exponent $q<+\infty$. We shall use the idea from   \cite{55Ph2} to obtain such a result.

\begin{lemma} \label{rei-es} Let $q>p$ and $x_1\in\partial\Omega$, $0<r\leq r_1\leq R_0/50$, and $v$ be as above. There exists $\delta=\delta(q)>0$ such that if $\Omega$ is a  $(\delta,R_0)$-Reifenberg flat domain then
	\begin{equation}\label{z8}
	\left(	\fint_{B_{r/800}(x_1)}|\nabla v|^q\right)^{1/q}\leq  C\left(	\fint_{B_{r}(x_1)}\left(|\nabla v|+s\right)^{\gamma_0}\right)^{\frac{1}{\gamma_0}}.
	\end{equation}
Here the constant $C$ does not depend on $r$.	In particular, for any $\varepsilon\in (0,1/800)$,
$$
		\fint_{B_{\varepsilon r}(x_1)}|\nabla v|^{\gamma_0}\leq  C\varepsilon^{-\frac{\gamma_0n}{q}}	\fint_{B_{r}(x_1)}\left(|\nabla v|+s\right)^{\gamma_0}.
$$
\end{lemma}

To prove Lemma \ref{rei-es}, we use the following lemma (see \cite[Theorem 3]{W03}). 

\begin{lemma}
	\label{lem:mainlem}
	Let $0<\epsilon<1$ and  $B_R$ be a ball of radius $R$ in $\mathbb{R}^n$.  Let $E\subset F\subset B_R$ be two measurable sets  with $|E|<\epsilon |B_R|$ and satisfy the following property: for all $x\in B_R$ and $\rho\in (0,R]$, we have $B_\rho(x)\cap B_R\subset F$  	provided $|E\cap B_\rho(x)|\geq \epsilon |B_\rho(x)|$.	Then $|E|\leq B\varepsilon |F|$ for some $B=B(n)$.
\end{lemma}
\begin{proof}[Proof of Lemma \ref{rei-es}] Assume that $\Omega$ is a  $(\delta,R_0)$-Reifenberg flat domain and $0<r\leq r_1\leq R_0/50$.
	
\noindent 	\underline{Step 1.}  Let ${\bf M}$ be the standard Hardy-Littlewood maximal function and write $\mathbf{1}_E$ to denote the characteristic function of a set $E$. Set $\rho=r/800$ and  for  $\lambda>0$ let 
$$E_{\lambda}=\left\{({\bf M}(\mathbf{1}_{B_{8\rho}(x_1)}|\nabla v|^{\gamma_0}))^{1/\gamma_0}>\lambda\right\}\cap B_{\rho}(x_1).$$

	 In this step, we show that for any $\epsilon>0$ one can find  constants $\delta_1=\delta_1(n,p,\Lambda,\epsilon)\in (0,1),\delta_2=\delta_2(n,p,\Lambda,\epsilon)\in (0,1)$ and $\Lambda_0=\Lambda_0(n,p,\gamma_0,\Lambda)>1$ such that if $\delta\leq \delta_1$, we have 
\begin{equation}|E_{\Lambda_0\lambda}|\leq C \epsilon |E_{\lambda}|\label{lam-good}
\end{equation}
	for any $\lambda\geq T_0$, where we define
 $$T_0:=\delta_2^{-1}\left(	\fint_{B_{800\rho}(x_1)} (|\nabla v|+s)^{\gamma_0}dx\right)^{\frac{1}{\gamma_0}}.$$

Since ${\bf M}$ is a bounded operator from $L^1(\mathbb{R}^{n})$ into $L^{1,\infty}(\mathbb{R}^{n})$,  we have for $\lambda\geq T_0$,
\begin{align}\label{5hh2310131}
\left|E_{\Lambda_0\lambda}\right|\leq \frac{C(n)}{(\Lambda_0\lambda)^{\gamma_0}}\int_{B_{8\rho}(x_1)} |\nabla v|^{\gamma_0}dx\leq C(n)(\delta_2/\Lambda_0)^{\gamma_0}|B_{800\rho}(x_1)|\leq \epsilon \left|B_\rho(x_1)\right|,
\end{align}
	provided $\delta_2\leq (800^{-n}\epsilon/C(n))^{1/\gamma_0}\Lambda_0$.

	 Next we verify that   for any $x\in B_{\rho}(x_1)$, $\rho_1\in(0,\rho]$ and $\lambda\geq  T_0$ we have
	\begin{equation}\label{2nd-check}
	\left|E_{\Lambda_0\lambda}\cap B_{\rho_1}(x)\right|\geq \epsilon \left|B_{\rho_1}(x)\right| \Longrightarrow B_{\rho_1}(x)\cap B_\rho(x_1)\subset E_\lambda,
	\end{equation}
	provided $\delta$ and $\delta_2$ are small enough depending on $n,p,\Lambda, \gamma_0,\epsilon$. Therefore, using \eqref{5hh2310131}-\eqref{2nd-check} and applying Lemma \ref{lem:mainlem} with $E= E_{\Lambda_0\lambda}$ and $F=E_\lambda$ we get \eqref{lam-good}.
	
	To prove \eqref{2nd-check},  take $x\in B_{\rho}(x_1)$, $\rho_1\in (0,\rho]$, and $\lambda\geq  T_0$, 
 and by contradiction, let us assume that $B_{\rho_1}(x)\cap B_\rho(x_1)\cap (E_\lambda)^c\not= \emptyset$,  i.e., there exists $x_2\in B_{\rho_1}(x)\cap B_\rho(x_1)$ such that 
\begin{equation}\label{x2lambda}
({\bf M}(\mathbf{1}_{B_{8\rho}(x_1)}|\nabla v|^{\gamma_0})(x_2))^{1/\gamma_0}\leq \lambda.
\end{equation}	

We need to prove that
	\begin{equation}\label{5hh2310133}
	\left|E_{\Lambda_0\lambda}\cap B_{\rho_1}(x)\right|<  \epsilon \left|B_{\rho_1}(x)\right|. 
	\end{equation}
	
	Clearly, for any $y\in B_{\rho_1}(x)$
$$
({\bf M}(\mathbf{1}_{B_{8\rho}(x_1)}|\nabla v|^{\gamma_0})(y))^{1/\gamma_0}\leq \max\left\{\left({\bf M}\left(\mathbf{1}_{B_{2\rho_1}(x)}|\nabla v|^{\gamma_0}\right)(y)\right)^{\frac{1}{\gamma_0}},3^{\frac{n}{\gamma_0}}\lambda\right\},$$
and thus 		 for all $\lambda\geq T_0$ and $\Lambda_0\geq 3^{\frac{n}{\gamma_0}}$,
\begin{equation}\label{inclusion}
 E_{\Lambda_0\lambda}\cap B_{\rho_1}(x)\subset\left\{\left({\bf M}\left(\mathbf{1}_{ B_{2\rho_1}(x)}|\nabla v|^{\gamma_0}\right)\right)^{\frac{1}{\gamma_0}}>\Lambda_0\lambda\right\}\cap B_{\rho}(x_1) \cap B_{\rho_1}(x).
\end{equation}	
	
	Now to prove \eqref{5hh2310133} we separately consider  the case $B_{4\rho_1}(x)\subset\subset \Omega$ and the case $\overline{B_{4\rho_1}(x)}\cap  \Omega^{c}\not=\emptyset$.
	
	\medskip
	
	\noindent {\bf 1. The case $B_{4\rho_1}(x)\subset\subset \Omega$:} Since $ \operatorname{div}\left( {A(x_1,\nabla v)} \right) = 0$ in $B_{4\rho_1}(x)$, by the standard regularity estimate, we have 
	$$
	||\nabla v||_{L^\infty(B_{3\rho_1}(x))}\leq C\left(	\fint_{B_{4\rho_1}(x)} (|\nabla v|+s)^{\gamma_0}\right)^{\frac{1}{\gamma_0}}\leq C_1\left(	\fint_{B_{5\rho_1}(x_2)} (|\nabla v|+s)^{\gamma_0}\right)^{\frac{1}{\gamma_0}}.$$
	
	Thus, using \eqref{x2lambda} and $s\leq \delta_2\lambda\leq \lambda$, we find
$$||\nabla v||_{L^\infty(B_{3\rho_1}(x))}\leq  C_1\left(\lambda+s\right)\leq 2 C_1\lambda.$$

Then for $\Lambda_0\geq \max\{3^{\frac{n}{\gamma_0}}, 4 C_1\}$, we have $	||\nabla v||_{L^\infty(B_{3\rho_1}(x))}\leq \frac{1}{2}\Lambda_0\lambda$ and so by \eqref{inclusion}
$E_{\Lambda_0\lambda}\cap B_{\rho_1}(x)=\emptyset.$ In particular, we have \eqref{5hh2310133}.

	\medskip
	
	\noindent {\bf 2. The case $\overline{B_{4\rho_1}(x)}\cap\Omega^{c}\not=\emptyset$:} Let $x_3\in\partial \Omega$ be such that $|x_3-x|=\text{dist}(x,\partial\Omega)$.  We have 
	$$B_{2\rho_1}(x)\subset B_{6\rho_1}(x_3)\subset B_{600\rho_1}(x_3)\subset B_{605\rho_1}(x_2).$$
	Thanks to  \cite[Proposition 2.6]{QH4}, (see also \cite[Corollary 2.13]{55Ph2}),	for any $\eta>0$ there exists $\delta_1=\delta_1(n,p,\Lambda,\eta)$ be such that the following holds. If $\delta\leq \delta_1$, there exists a function $\tilde{v}\in W^{1,\infty}(B_{6\rho_1}(x_3))$  such that $$
	\|\nabla \tilde{v}\|_{L^\infty(B_{6\rho_1}(x_3))}\leq C_0 \left(\fint_{B_{600\rho_1}(x_3)}(|\nabla v|+s)^{\gamma_0}\right)^{1/\gamma_0},$$
	and		$$
	\left(\fint_{B_{6\rho_1}(x_3)}|\nabla (v-\tilde{v})|^{\gamma_0}\right)^{\frac{1}{\gamma_0}}\leq  \eta\left(\fint_{B_{600\rho_1}(x_3)}(|\nabla v|+s)^{\gamma_0}\right)^{1/\gamma_0}.$$
	
	Note that if $\rho_1\leq \rho/100$, then 
	$$\left(\fint_{B_{600\rho_1}(x_3)}(|\nabla v|+s)^{\gamma_0}\right)^{1/\gamma_0}\leq 2^{\frac{n}{\gamma_0}} ({\bf M}(\mathbf{1}_{B_{8\rho}(x_1)}|\nabla v|^{\gamma_0})(x_2))^{1/\gamma_0} +s\leq 2^{\frac{n}{\gamma_0}+1}\lambda,$$
	 and if $\rho_1\geq  \rho/100$, then since $\rho_1\leq \rho$,
		$$\left(\fint_{B_{600\rho_1}(x_3)}(|\nabla v|+s)^{\gamma_0}\right)^{1/\gamma_0}\leq10^{\frac{3n}{\gamma_0}}\left(\fint_{B_{800\rho}(x_1)}(|\nabla v|+s)^{\gamma_0}\right)^{1/\gamma_0}\leq 10^{\frac{3n}{\gamma_0}} \delta_2\lambda. $$
	
      Hence, 
	$$
	\|\nabla \tilde{v}\|_{L^\infty(B_{2\rho_1}(x))}\leq 10^{\frac{3n}{\gamma_0}}C_0 \lambda,$$
and
 	$$\left(\fint_{B_{2\rho_1}(x)}|\nabla (v-\tilde{v})|^{\gamma_0}\right)^{\frac{1}{\gamma_0}}\leq 10^{\frac{4n}{\gamma_0}}\eta \lambda.$$
	
Choosing $\Lambda_0= \max\{3^{\frac{n}{\gamma_0}},4 C_1, 2^{\frac{1}{\gamma_0}} 10^{\frac{3n}{\gamma_0}} C_0\}$, we have 
	\begin{align*}
	 |E_{\Lambda_0\lambda}\cap B_{\rho_1}(x)|&\leq \left|\left\{\left({\bf M}\left(\mathbf{1}_{ B_{2\rho_1}(x)}|\nabla (v-\tilde{v})|^{\gamma_0}\right)\right)^{\frac{1}{\gamma_0}}>2^{-\frac{1}{\gamma_0}}\Lambda_0\lambda\right\}\right|\\&\leq  \frac{C(n)}{\left(2^{-\frac{1}{\gamma_0}}\Lambda_0\lambda\right)^{\gamma_0}} \int_{B_{2\rho_1}(x)}|\nabla (v-\tilde{v})|^{\gamma_0}
	 \\&\leq  \frac{2 C(n)}{\left(\Lambda_0\lambda\right)^{\gamma_0}} \left(10^{\frac{4n}{\gamma_0}}\eta \lambda\right)^{\gamma_0}|B_{2\rho_1}(x)|\\&< \epsilon \left|B_{\rho_1}(x)\right|,
	\end{align*}
	for $\eta=\left(\epsilon/(10^{5n}C(n))\right)^{1/\gamma_0}$. This gives \eqref{5hh2310133}.

	 \medskip
	 
\noindent	\underline{Step 2.} Thanks to  \eqref{lam-good},
	we have that for $\lambda_0=\Lambda_0 T_0$,
	\begin{align*}
	&\int_{B_{\rho}(x_1)}	({\bf M}(\mathbf{1}_{B_{8\rho}(x_1)}|\nabla v|^{\gamma_0}))^{q/\gamma_0} dx\\&= q\int_{0}^{\infty}\lambda^{q-1}\left|\left\{ |({\bf M}(\mathbf{1}_{B_{8\rho}(x_1)}|\nabla v|^{\gamma_0}))^{1/\gamma_0}>\lambda\right\}\cap B_{\rho}(x_1)\right|d\lambda
	\\ & \leq  q\int_{0}^{\lambda_0}\lambda^{q-1}|B_{\rho}(x_1)|d\lambda\\&+ 
	Cq\epsilon\int_{\lambda_0}^{\infty}\lambda^{q-1}\left|\left\{|({\bf M}(\mathbf{1}_{B_{8\rho}(x_1)}|\nabla v|^{\gamma_0}))^{1/\gamma_0}>\lambda/\Lambda_0\right\}\cap B_{\rho}(x_1)\right|d\lambda\\&\leq  \lambda_0^{q}|B_{\rho}(x_1)| +C\Lambda_0^q\epsilon \int_{B_{\rho}(x_1)}	({\bf M}(\mathbf{1}_{B_{8\rho}(x_1)}|\nabla v|^{\gamma_0}))^{q/\gamma_0} dx.
	\end{align*}
	
Thus 	letting $\epsilon=\frac{1}{2 C\Lambda_0^q}$, we get \footnote{ A limiting argument can be used to justify that $\int_{B_{\rho}(x_1)}	({\bf M}(\mathbf{1}_{B_{8\rho}(x_1)}|\nabla v|^{\gamma_0}))^{q/\gamma_0} dx$ is finite.}
		\begin{equation*}
\fint_{B_{\rho}(x_1)}|\nabla v|^{q}dx\leq C\, T_0^q =   C\left(	\fint_{B_{800\rho}(x_1)} (|\nabla v|+s)^{\gamma_0}dx\right)^{\frac{q}{\gamma_0}}.
	\end{equation*}
	
	Now recall that $\rho=r/800$ and hence \eqref{z8} follows. This completes the proof of the theorem.	
\end{proof}

The following technical lemma can be found in \cite[Lemma 3.4]{HL}. 

\begin{lemma}\label{lem00}
	Let $\phi$ be a nonnegative and nondecreasing functions on $(0,D]$. Suppose that there are nonnegative constants $A,B,\alpha,\beta$ with $\alpha>\beta$ such that 
	\begin{align*}
	\phi(\rho)\leq A\left[(\rho/R)^\alpha+\eta\right]\phi(R)+BR^\beta,
	\end{align*} for all $0<\rho\leq R\leq D$.	
	Then for any $\gamma\in [\beta,\alpha)$, there exits positive $\eta_0=\eta_0(\alpha,\beta, \gamma,A)$ such that if $\eta\leq \eta_0$ we have 
	\begin{align*}
	\Phi(\rho)\leq C (\rho/R)^{\gamma}\Phi(R)+ C B \rho^\beta,
	\end{align*}
	for all $0<\rho\leq R\leq D$. Here $C=C(\alpha,\beta, \gamma,A)$.
\end{lemma}
 
We are now ready to finish the proof of Theorem \ref{boundary}. Let $\kappa\in (0,1/2)$ be fixed.  By Lemma \ref{rei-es}, there  exists $\delta=\delta(\kappa)>0$ such that if $\Omega$ is a  $(\delta,R_0)$-Reifenberg flat domain 
then 
	\begin{align*}
		\int_{B_{\varepsilon r}(x_1)}|\nabla v|^{\gamma_0}\leq C	\varepsilon^{n-\gamma_0\kappa/2}\int_{B_{ r}(x_1)}|\nabla v|^{\gamma_0} 
	\end{align*}
	for all $r\leq r_1$ and $\epsilon\in (0,1/800)$. 
	Writing  $B_r=B_r(x_1)$, we  thus have  
\begin{align}\nonumber
	\int_{B_{\varepsilon r}} &|\nabla u|^{\gamma_0}\\
	&\leq c	\int_{B_{\varepsilon r}}|\nabla v|^{\gamma_0}+ c	\int_{B_{\varepsilon r}}|\nabla v-\nabla w|^{\gamma_0}+ c\int_{B_{\varepsilon r}}|\nabla u-\nabla w|^{\gamma_0}\nonumber\\ & \nonumber\leq c\,	\varepsilon^{n-\gamma_0\kappa/2}\int_{B_{ r}}|\nabla v|^{\gamma_0}+ c	\int_{B_{r}}|\nabla v-\nabla w|^{\gamma_0}+c\int_{B_{r}}|\nabla u-\nabla w|^{\gamma_0}
	\\ & \leq c\,	\varepsilon^{n-\gamma_0\kappa/2}\int_{B_{ r}}|\nabla u|^{\gamma_0}+ c	\int_{B_{r}}|\nabla v-\nabla w|^{\gamma_0}+ c\int_{B_{r}}|\nabla u-\nabla w|^{\gamma_0}.\label{es10}
\end{align}

At this point, we use  Lemma \ref{111120149+} to bound the last term in \eqref{es10} and  use \eqref{es-1} to bound the seccond to last term in \eqref{es10}.
This gives  for any $\epsilon, \eta \in (0,1/800)$,
\begin{align*}
\int_{B_{\varepsilon r}}|\nabla u|^{\gamma_0} & \leq C	\left(\varepsilon^{n-\gamma_0\kappa/2}+ \omega(r)^{\gamma_0} \right)\int_{B_{ 2r}}(|\nabla u|+s)^{\gamma_0}+  C r^n\left[	\frac{|\mu|(B_{2r})}{r^{n-1}}\right]^{\frac{\gamma_0}{ p-1}}\\&~~+C r^{n(p-1)}	\left(\frac{|\mu|(B_{2r})}{r^{n-1}}\right)^{\gamma_0}\left(	\int_{B_{2r}}(|\nabla u|+s)^{\gamma_0}dx\right)^{2-p}\\
& \leq C	\left(\varepsilon^{n-\gamma_0\kappa/2}+ \omega(r)^{\gamma_0}+\eta \right)\int_{B_{ 2r}}(|\nabla u|+s)^{\gamma_0}+  C_{\eta}\, r^n\left[	\frac{|\mu|(B_{2r})}{r^{n-1}}\right]^{\frac{\gamma_0}{ p-1}}. 
\end{align*}
Here we use Young's inequality in the last inequality. Note that this holds for any $r\in(0,r_1] $ and by enlarging $C$ if necessary it also holds for any $\epsilon\in (0,2)$.
Thus we find
\begin{align*}
\int_{B_{\rho}(x_1)}|\nabla u|^{\gamma_0}
& \leq C	\left((\rho/R)^{n-\gamma_0\kappa/2}+ \omega(r_1)^{\gamma_0}+\eta \right)\int_{B_{R}(x_1)}|\nabla u|^{\gamma_0}\\
&\quad + C_\eta \, R^{n-\gamma_0\kappa} r_1^{\gamma_0\kappa} \left(\Big[{\bf P}^{2 {\rm diam}(\Omega)}_{\gamma_0}(|\mu|)(x_0)\Big]^{\frac{1}{p-1}}+s^{\gamma_0}\right),
\end{align*}
for all $0<\rho\leq R\leq 2r_1$.

Now applying Lemma \ref{lem00} to $\phi(r)=\int_{B_{r}(x_1)}|\nabla u|^{\gamma_0}$, $r\in(0,2r_1)$, we obtain 
\begin{align*}
\int_{B_{\rho}(x_1)}|\nabla u|^{\gamma_0}
& \leq C(\rho/R)^{n-\gamma_0\kappa}\int_{B_{R}(x_1)}|\nabla u|^{\gamma_0}\\
&\quad +  C \, \rho^{n-\gamma_0\kappa} r_1^{\gamma_0\kappa} \left(  \Big[{\bf P}^{2 {\rm diam}(\Omega)}_{\gamma_0}(|\mu|)(x_0)\Big]^{\frac{1}{p-1}} +s^{\gamma_0}\right),
\end{align*}
provided that $\omega(r_1)$ and $\eta$ are small enough. In particular, for $R=2r_1$ and $\rho=2 d(x_0)$ we find 
\begin{align*}
&\fint_{B_{2d(x_0)}(x_1)}  |\nabla u|^{\gamma_0}\\
& \quad \leq C\left(\frac{r_1}{d(x_0)}\right)^{\gamma_0\kappa}\left(\fint_{B_{ 2 r_1}(x_1)}|\nabla u|^{\gamma_0}+  \Big[{\bf P}^{2 {\rm diam}(\Omega)}_{\gamma_0}(|\mu|)(x_0)\Big]^{\frac{1}{p-1}} +s^{\gamma_0}\right).
\end{align*}

This implies 
\begin{align*}
&\fint_{B_{d(x_0)}(x_0)}  |\nabla u|^{\gamma_0}\\
& \quad \leq C\left(\frac{r_1}{d(x_0)}\right)^{\gamma_0\kappa}\left(\fint_{B_{ 2 r_1}(x_1)}|\nabla u|^{\gamma_0}+  \Big[{\bf P}^{2 {\rm diam}(\Omega)}_{\gamma_0}(|\mu|)(x_0)\Big]^{\frac{1}{p-1}} +s^{\gamma_0}\right)\\
& \quad \leq C(r_1)\,  d(x_0)^{-\gamma_0\kappa}   \left(\Big[{\bf P}^{2 {\rm diam}(\Omega)}_{\gamma_0}(|\mu|)(x_0)\Big]^{\frac{1}{p-1}} +s^{\gamma_0}\right),
\end{align*}
where we used \eqref{int0} in the last inequality.

Now applying this result to \eqref{Wwithtail} we arrive at  \eqref{ine3}. This completes the proof of Theorem \ref{boundary}. 
\begin{remark} Our argument  works also in the case $p>2-\frac{1}{n}$ provided we use
	 the local interior pointwise gradient estimates obtained in the  work \cite{Duzamin2, KM}. In this case, of course the truncated 
	 Riesz's potential ${\bf I}^{2{\rm diam}(\Omega)}_{1}(|\mu|)$ is used in placed of	 ${\bf P}^{2{\rm diam}(\Omega)}_{\gamma_0}(|\mu|)^{1/\gamma_0}$. 
\end{remark}


\begin{thebibliography}{99}
	\bibitem{AF} E. Acerbi and N. Fusco, 	{\em Regularity for minimizers of non-quadratic functionals: The case $1<p<2$,} J. Math. Anal. Appl. {\bf 140} (1989), 115--135.
	\bibitem{bebo} P. B\'enilan, L. Boccardo, T. Gallou\"et, R. Gariepy, M. Pierre, and J. L. Vazquez,
	{\em An $L^1$ theory of existence and uniqueness of solutions of nonlinear elliptic equations,} Ann. Scuola Norm. Sup. Pisa (4) {\bf 22} (1995), 241--273.
	\bibitem {11DMOP} G. Dal Maso, F. Murat, L. Orsina, and A. Prignet, {\em Renormalized solutions of elliptic equations with general measure data},  Ann. Scuola Norm. Sup. Pisa (4) {\bf 28} (1999), 741--808. 
	\bibitem{DK} L. Diening and Ch. Kreuzer, {\it Linear convergence of an adaptive finite element method for the $p$-Laplacian equation}, SIAM J. Numer. Anal. {\bf 46} (2008), 614--638.
	\bibitem{DSV} L. Diening, B. Stroffolini, and A. Verde, {\em Everywhere regularity of functionals with $\varphi$-growth}, Manuscripta Math. {\bf 129}  (2009), 449--481.
	\bibitem{Diening} L. Diening, P. Kaplick\'y, and S. Schwarzacher, {\em  BMO estimates for the p-Laplacian},  Nonlinear Anal. {\bf 75} (2012), 637--650.
	\bibitem{Duzamin1} F. Duzaar and G. Mingione, {\it Local Lipschitz regularity for degenerate elliptic systems}, Ann. Inst. H. Poincar\'e Anal. Non Lin\'eaire {\bf 27} (2010),  1361--1396.
	\bibitem{Duzamin2} F. Duzaar and G. Mingione, {\em Gradient estimates via linear and nonlinear potentials}, 
	J. Funt. Anal. {\bf 259} (2010), 2961--2998.
	\bibitem{55DuzaMing} F. Duzaar and G. Mingione, {\em Gradient estimates via non-linear potentials},  Amer. J. Math. {\bf133} (2011), 1093--1149. 
	\bibitem{GM} M. Giaquinta and G. Modica, {\em Remarks on the regularity of the minimizers of certain degenerate functionals}, Manuscripta Math. {\bf 57} (1986), 55--99.
	\bibitem{Giu}E. Giusti, {\em Direct methods in the calculus of variations}, World Scientic Publishing Co., Inc., River Edge, NJ, 2003.
\bibitem{HL} Q. Han and F. Lin, {\it Elliptic partial differential equations}, Second Edition. Courant
Lecture Notes in Mathematics, Vol. 1. Courant Institute of Mathematical Sciences,
New York; American Mathematical Society, Providence, RI, 2011. x+147 pp.

	\bibitem{55KeTo1} C. Kenig and  T. Toro, {\em Free boundary regularity for harmonic measures and the Poisson kernel}, Ann. Math. {\bf 150}  (1999), 367--454.
	\bibitem{55KeTo2} C. Kenig and  T. Toro, {\em Poisson kernel characterization of Reifenberg flat chord arc domains}, Ann. Sci. \'Ecole 
	Norm. Sup. {\bf 36} (2003), 323--401.
	\bibitem{Mi2} G. Mingione, {\em The Calder\'on-Zygmund theory for elliptic problems with measure data}, Ann. Scuola Norm. Sup. Pisa Cl. Sci. (5) {\bf 6} (2007), 195--261.
	
	\bibitem{KM} T. Kuusi and G.  Mingione,  {\it Linear potentials in nonlinear potential theory}, Arch. Ration. Mech. Anal. {\bf 207} (2013),  215--246.
	 
	\bibitem{QH4} Q.-H. Nguyen and N. C. Phuc, {\em Good-$\lambda$  and Muckenhoupt–Wheeden type bounds in quasilinear measure datum problems, with applications},  Math. Ann. (2018). https://doi.org/10.1007/s00208-018-1744-2
	\bibitem{55Ph2} N. C. Phuc, {\em Nonlinear Muckenhoupt-Wheeden type bounds on Reifenberg flat domains, with applications to quasilinear Riccati type equations,} Adv. Math. {\bf 250} (2014), 387--419.
	\bibitem{55Re} E. Reifenberg, {\em Solutions of the Plateau Problem for m-dimensional surfaces of varying topological type}, Acta Math. {\bf 104} (1960), 1--92.
	\bibitem{W03} L. Wang,  {\it A geometric approach to the Calder\'on-Zygmund estimates}, Acta Math. Sin. (Engl.
	Ser.) {\bf 19} (2003), 381--396.
	
	
\end{thebibliography}
\end{document}